\documentclass[11pt,a4paper,reqno]{amsart}
\usepackage[utf8]{inputenc}
\usepackage[T1]{fontenc}
\usepackage[foot]{amsaddr}
\usepackage{amsmath}
\usepackage{amsfonts,dsfont}
\usepackage{amssymb}
\usepackage{algorithm}
\usepackage{algpseudocode}
\usepackage{xcolor}
\usepackage{soul}
\usepackage{graphicx}
\usepackage{wrapfig}
\usepackage{subcaption}
\usepackage[nopar]{lipsum}
\usepackage{amsthm}
\usepackage{amssymb}
\usepackage{euscript}
\usepackage{blkarray}
\usepackage{makecell, multirow, tabularx,booktabs}
\usepackage{enumitem}
\usepackage{mathtools}
\usepackage{soul}
\usepackage{amscd}
\usepackage{float}
\usepackage[pagebackref,colorlinks=true,
linkcolor=blue,
urlcolor=red,colorlinks,
citecolor=green]{hyperref}

\usepackage{cleveref}
\usepackage{geometry}
\usepackage{pgf,tikz,pgfplots}
\usepackage{tikz-cd}
\usepackage{adjustbox}
\usepackage{setspace}
\usepackage[numbers,sort&compress]{natbib}

\allowdisplaybreaks[4]
\usetikzlibrary{calc}
\usetikzlibrary{patterns}
\pgfplotsset{compat=1.14}
\numberwithin{equation}{section}
\newtheorem{thm}{Theorem}[section]
\newtheorem{prop}[thm]{Proposition}
\newtheorem{lem}[thm]{Lemma}
\newtheorem{cor}[thm]{Corollary}
\newtheorem{exm}[thm]{Example}
\newtheorem{df}[thm]{Definition}

\newcommand {\n} {\mathbb{N}}

\newcommand{\q}{Q_{(H,\delta_{V(H)},\delta_{E(H)})}}
\newcommand{\lap}{L_{(H,\delta_{V(H)},\delta_{E(H)})}}
\newcommand{\adj}{A_{(H,\delta_{V(H)},\delta_{E(H)})}}

\tikzstyle{none}=[]
\tikzstyle{new style 0}=[draw,circle,fill=white]
\tikzstyle{new edge style 1}=[draw,dashed]
\tikzstyle{new edge style 1}=[draw,dashed]
\tikzstyle{new edge style 0}=[->]

\def\mf{\mathfrak}

\def\Lh{Let $H$ be a hypergraph. }

\makeatletter
\@namedef{subjclassname@1991}{Subject}
\@namedef{subjclassname@2000}{Subject}
\@namedef{subjclassname@2010}{Subject}
\@namedef{subjclassname@2020}{Subject}
\makeatother
\begin{document}
\title[Some building blocks of hypergraphs]{On some building blocks of hypergraphs}
	\author[Banerjee]{Anirban Banerjee} 
	\email[Banerjee ]{\textit {{\scriptsize anirban.banerjee@iiserkol.ac.in}}}
	\author[ Parui]{Samiron Parui} 
	\email[ Parui ]{\textit {{\scriptsize  samironparui@gmail.com}}}

	\address[Banerjee, Parui]{Department of Mathematics and Statistics, Indian Institute of Science Education and Research Kolkata, Mohanpur-741246, India}

	
	\date{\today}
	\vspace{-5pt}
	\maketitle
\hrule
\begin{abstract}
  In this study, we explore the interrelation between hypergraph symmetries represented by equivalence relations on the vertex set and the spectra of operators associated with the hypergraph. 
  We introduce the idea of equivalence relation compatible operators related to hypergraphs. 
  Some eigenvalues and the corresponding eigenvectors can be computed directly from the equivalence classes of the equivalence relation. The other eigenvalues can be computed from a quotient operator obtained by identifying each equivalence class as an element. The equivalence classes of specific equivalence relations on the vertex set determine the structures of the hypergraph. We collectively classify them as building blocks of hypergraphs. For instance, an equivalence relation $\mf{R}_s$ on the vertex set of a hypergraph is such that two vertices are $\mathfrak{R}_s$-related if they belong to the same set of hyperedges. 
  The $\mf{R}_s$-equivalence classes are named as units. Using units, we explore another symmetric substructure of hypergraphs called twin units. We show that these building blocks leave certain traces in the spectrum and the corresponding eigenspaces of the $\mf{R}_s$-compatible operators associated with the hypergraph. We also show that, conversely, some specific footprints in the spectrum and the corresponding eigenvectors retrace the presence of some of these building blocks in the hypergraph. Besides the spectra of $\mf{R}_s$-compatible operators, building blocks are also interrelated with hypergraph colouring, distances in hypergraphs, hypergraph automorphisms, and random walks on hypergraphs.
\end{abstract}
\vspace{5pt}
\hrule
\section{Introduction}
The interrelation between the structural symmetry of a graph and the spectra of matrices associated with the graph is one of the classical directions in spectral graph theory. Symmetries in the graph structures are represented using automorphisms, equitable partitions, divisors of a graph, quotient matrices, twin vertices, etc.~and the manifestation of these symmetric structures in the spectra of matrices associated with the graph are documented in \cite{Steve-Butler-2010-edgecover,Trevisan_symm_simax_2016,Oscar-rojo-balanced-tree-2005,mehatari-banerjee-motif2015,Vladimir-R-Covering-automorphisms,Benjamin-webb-gen-equitable-2019,Cvetkovic-1980}. 
 In an equitable partition of a graph $G$, identifying all the vertices within each part of the equitable partition leads to a quotient matrix. This quotient matrix corresponds to a
smaller graph (possibly a directed multi-graph or a weighted graph), called a divisor of the graph $G$ (\cite[Section 4.1]{Cvetkovic-1980}). Each eigenvalue of the quotient matrix is also an eigenvalue of the adjacency matrix of $G$(\cite[Theorem 4.8.]{Cvetkovic-1980}). A similar instance is evident in an equitable decomposition of a graph due to a graph automorphism (\cite{Benjamin-webb-equitable-2016,Benjamin-webb-gen-equitable-2019}). In an equitable decomposition, besides the quotient matrix, automorphism orbits induce smaller matrices whose eigenvalues are also the eigenvalues of a matrix associated with the original graph. In this paper, we 
use equivalence relation on the vertex set to capture hypergraph symmetries. Some eigenvalues of matrices associated with hypergraphs can be obtained from the equivalence classes, while the rest of the eigenvalues are the eigenvalues of the quotient matrix obtained by the identification of vertices in the same equivalence class.

A hypergraph $H$ is an ordered pair  $(V(H), E(H))$, where the \emph{vertex set}, $V(H)$ is a non-empty set and the \emph{hyperedge set}, $E(H)$ is a collection of non-empty subsets of $V(H)$. Various matrices, such as adjacency matrix, Laplacian matrix, signless Laplacian matrix, etc., are used to investigate hypergraphs (\cite{hg-mat,MR4107825,up2021,rodriguez2003Laplacian,cardoso2019signless,feng1996spectra}). In this study, we delve into the interplay between the structure of hypergraphs and the spectral properties exhibited by various hypergraph matrices, including the adjacency matrix, Laplacian matrix, and signless Laplacian matrix. Our approach involves delineating distinct substructures within hypergraphs through the lens of equivalence relations on vertex sets. We show that some specific hypergraph substructures lead to particular eigenvalues with specific eigenspaces of hypergraph matrices. 
The incidence structure of a vertex $v\in V(H)$ is represented by the \emph{star of the vertex} $v$, $E_v(H)=\{e\in E(H):v\in e\}$. A maximal collection of vertices that have the same star is called a \emph{unit} in $H$. Each unit with at least two vertices corresponds to a particular eigenvalue of specific hypergraph matrices.

If symmetric substructures of a hypergraph induce symmetry in an operator associated with that hypergraph, they
 leave their traces in the spectrum of that operator. 
 We decompose a hypergraph into symmetric substructures using equivalence relations. 
 These hypergraph symmetries, related to an equivalence relation, are encoded in a set of matrices (or operators) associated with the hypergraph. We refer to this set of operators as the \emph{equivalence relation compatible operators} (see \Cref{df_eq_cmpt_op}). 
 The complete spectrum of equivalence relation compatible operators associated with hypergraphs can be described in terms of 
 the equivalence relation. Suppose $H$ is a hypergraph, and $\mf{R}$ is an equivalence relation on $V(H)$.
 The complete spectrum of an $\mf{R}$-compatible operator associated with a hypergraph can be computed in two parts. The first part is associated with the local symmetries represented by each equivalence class (\Cref{equivalence-eig}, \Cref{Tw}).
 In the second part, the remaining eigenvalues can be computed from the quotient operator obtained from the equivalence relation (\Cref{gen-contract-eig}). These local symmetries, represented by equivalence class, determine the structure of hypergraphs. We classify them as {building blocks of hypergraphs}. \emph{Unit} is one of these building blocks.
 Using our introduced methods for the spectral study of the equivalence relation compatible operators, in \Cref{1.unit}, we provide an equivalence relation $\mf{R}_s$ (see \Cref{df_unit}) such that the general signless Laplacian, Laplacian, and adjacency associated with a hypergraph are compatible with the equivalence relation $\mf{R}_s$. We refer to each $\mf{R}_s$-equivalence class as a \emph{unit}. Units in a hypergraph leave traces in the spectra of adjacency, Laplacian, and signless Laplacian associated with the hypergraph (\Cref{cor_unit_eig}, \Cref{com-unil-a}, and \Cref{com-unil-a}). Conversely, some information about the units is encoded in the spectra of these operators (\Cref{cor_unit_eig}). 
This section concludes by revisiting certain basic ideas that will be frequently referenced throughout the subsequent sections. Let $H$ be a hypergraph.
For $v_1,\ldots,v_k\in V(H)$, the \emph{star of $v_1,\ldots,v_k $} is $E_{v_1,\ldots,v_k}(H)=E_{v_1}(H)\cap\ldots\cap E_{v_k}(H)$, and for any $U\subseteq V(H)$, we define the \emph{star of $U$} is $E_U(H)$. 
Suppose that $V$ is a non-empty finite set. The set of all complex-valued functions on $V$, denoted by $\mathbb{C}^V$, is a vector space of dimension $|V|
	$. For any $x\in\mathbb{C}^{V}$, \textit{the support of} $ x$  is $supp(x)=\{v\in V:x(v)\ne 0\} $. For any none empty $U\subseteq V$ the characteristic function $\chi_{U}:V(H)\to\{0,1\}$ is defined as $\chi_U(v)=1 $ if $v\in U$, and $\chi_U(v)=0$ if $v\notin u$.
A hypergraph is called a \emph{finite hypergraph} if $V(H)$ is a finite set. Here, we consider only finite hypergraphs. 
We denote the set of all the functions on $V(H)$ and $E(H)$ as $\mathbb{C}^{V(H)}$ and $\mathbb{C}^{E(H)}$, respectively. If $H$ is a hypergraph with $n$ vertices, and $m$ hyperedges then  $\mathbb{C}^{V(H)}$ and $\mathbb{C}^{E(H)}$, are isomorphic to respectively, $\mathbb{C}^n$, and $\mathbb{C}^m$. Thus, using an ordering of the vertex set $V(H)=\{v_1,\ldots, v_n\}$ each function $x\in \mathbb{C}^{V(H)}$ can be represented as a column vector $x\in\mathbb{C}^n$ with its $i$-th component $x_i=x(v_i)$. Here, we use the function and the column vector form of $x\in \mathbb{C}^{V(H)}$ as our convenience.

For any $u,v(\ne u)\in V(H)$, the function $x_{uv}:V(H)\to \{0,1\}$ is defined by $x_{uv}=\chi_{\{v\}}-\chi_{\{u\}}$.
 For any $x\in \mathbb{C}^{V(H)}$, the \emph{support of $x$} is $supp(x)=\{v\in V(H):x(v)\ne 0\}.$
For instance, the support of $x_{uv} $ is $supp(x_{uv})=\{u,v\}$, and $\sum\limits_{w\in supp(x_{uv})}x_{uv}(w)=0$.
For $U\subseteq V(H)$, we define the vector space $S_U$ as 
$S_U=\{x\in\mathbb{C}^{V(H)}:supp(x)\subseteq U, \text{~and~} \sum\limits_{v\in U}x(v)=0\} .$ If we enumerate $U$ as $U=\{u_0,\ldots,u_k\}$, then $\{x_{u_0u_i}:i=1,\ldots,k\}$ is a basis of $S_U$. Thus, the dimension of $S_U$ is $|U|-1$.

Suppose that $\delta_{V(H)}:V(H)\to (0,\infty)$, and $\delta_{E(H)}:E(H)\to (0,\infty)$ are two positive valued functions. The general operators are defined in \cite{up2021} using these positive valued functions as follows.
The general signless Laplacian operator $Q_{(H,\delta_{V(H)},\delta_{E(H)})}:\mathbb{C}^{V(H)}\to \mathbb{C}^{V(H)}$ is defined as
$$(Q_{(H,\delta_{V(H)},\delta_{E(H)})}x)(v)=\sum\limits_{e\in E_v(H)}\frac{\delta_{E(H)}(e)}{\delta_{V(H)}(v)}\frac{1}{|e|^2}\sum\limits_{u\in e}x(u),$$
and the general Laplacian operator $L_{(H,\delta_{V(H)},\delta_{E(H)})}:\mathbb{C}^{V(H)}\to \mathbb{C}^{V(H)}$ is defined as
$$ (L_{(H,\delta_{V(H)},\delta_{E(H)})}x)(v)=\sum\limits_{e\in E_v(H)}\frac{\delta_{E(H)}(e)}{\delta_{V(H)}(v)}\frac{1}{|e|^2}\sum\limits_{u\in e}(x(v)-x(u)).$$
The general adjacency operator $A_{(H,\delta_{V(H)},\delta_{E(H)})}:\mathbb{C}^{V(H)}\to \mathbb{C}^{V(H)}$ is given by
$$	(A_{(H,\delta_{V(H)},\delta_{E(H)})}x)(v)=\sum\limits_{e\in E_v(H)}\frac{\delta_{E(H)}(e)}{\delta_{V(H)}(v)}\frac{1}{|e|^2}\sum\limits_{u\in e;u\neq v}x(u).$$
We use these general operators for our spectral investigation because, for different choices of $(\delta_{V(H)},\delta_{E(H)})$, these operators become different variations of the signless Laplacian, Laplacian and adjacency operators associated with hypergraphs. Thus, the study on the general operators does the same for all their variations.
Here we consider the functions $\sigma_H:E(H)\to(0,\infty)$ and $ \rho_H:E(H)\to(0,\infty)$ defined as $\sigma_H(e)=\frac{\delta_{E(H)}(e)}{|e|^2}$ and $\rho_H(e)=\frac{\delta_{E(H)}(e
		)}{|e|}$  respectively.

\begin{exm}[Particular cases of general operators]\label{cases}\rm
    Some particular cases of the general operators that already exist in the literature are listed below. \begin{enumerate}[leftmargin=*]
\item If we take $\delta_{V(H)}(v)=1$ for all $v\in V(H)$, and $\delta_{E(H)}(e)=|e|^2$ for all $e\in E(H)$, then the operator $L_{(H,\delta_{V(H)},\delta_{E(H)})}$ becomes the Laplacian matrix, described in \cite{rodriguez2003Laplacian,rodriguez2009Laplacian}. Here, for the choice of  $\delta_{V(H)}(v)=1$ and $\delta_{E(H)}(e)=|e|^2$, we denote 
$Q_{(H,\delta_{V(H)},\delta_{E(H)})},L_{(H,\delta_{V(H)},\delta_{E(H)})}$, and $A_{(H,\delta_{V(H)},\delta_{E(H)})}$ as $Q_{(R,H)},L_{(R,H)}$, and $A_{(R,H)}$, respectively.
    \item If we choose $\delta_{V(H)}(v)=1$ for all $v\in V(H)$ and $\delta_{E(H)}(e)=\frac{|e|^2}{|e|-1} $ for all $e\in E(H)$, the operators $L_{(H,\delta_{V(H)},\delta_{E(H)})}$, and $A_{(H,\delta_{V(H)},\delta_{E(H)})}$ become the operators described in \cite{MR4208993,MR4107825,banerjee2020synchronization}. For this choice of $\delta_{V(H)}$ and $\delta_{E(H)}$, we denote $Q_{(H,\delta_{V(H)},\delta_{E(H)})},L_{(H,\delta_{V(H)},\delta_{E(H)})}$, and $A_{(H,\delta_{V(H)},\delta_{E(H)})}$ as $Q_{(B,H)},L_{(B,H)}$, and $A_{(B,H)}$, respectively.
    \item For the same choice of $\delta_{E(H)}$, if we set $\delta_{V(H)}(v)=|E_v(H)|$, for all $v\in V(H)$, then $L_{(H,\delta_{V(H)},\delta_{E(H)})}$ becomes the normalized Laplacian described in \cite{MR4208993}. For this choice of $\delta_{V(H)}, \delta_{E(H)}$, we denote $Q_{(H,\delta_{V(H)},\delta_{E(H)})},L_{(H,\delta_{V(H)},\delta_{E(H)})}$, and $A_{(H,\delta_{V(H)},\delta_{E(H)})}$ as $Q_{(N,H)},L_{(N,H)}$, and $A_{(N,H)}$, respectively.
\end{enumerate}
\end{exm}
 Let $T:\mathbb{C}^{V(H)}\to\mathbb{C}^{V(H)}$ be a linear operator associated with a hypergraph $H$. If for all $x:V(H)\to \mathbb{C}$, and for any $u\in V(H)$ the value $(Tx)(u)=\sum_{v\in V(H)}t_{uv}x(v)$, then the operator $T$ can be represented as the matrix $T=[t_{uv}]_{u,v\in V(H)}$. The matrix representation of $T$ is associated with a basis of $\mathbb{C}^{V(H)}$. Unless otherwise is mentioned, here, matrix representation of $T$ means matrix representation with respect to the basis $\mathcal{B}=\{\mathds{1}_v:v\in V(H)\}$, where for all $v\in V(H)$, the vector $\mathds{1}_v:V(H)\to\{0,1\}$ is defined by $\mathds{1}_v(u)=1$ if $u=v$, and $\mathds{1}_v(u)=0$ for all $u\in V(H)\setminus \{v\}$. For a finite graph $H$, a linear operator associated with $H$ and a matrix associated with $H$ are equivalent notions. In this work, we will use both, adapting to the specific requirements of different situations. For instance, the matrix representation of $Q_{(H,\delta_{V(H)},\delta_{E(H)})}$ is $Q_{(H,\delta_{V(H)},\delta_{E(H)})}=[q_{uv}]_{u,v\in V(H)}$ is such that $q_{uv}=\sum\limits_{e\in E_u(H)\cap E_v(H)}\frac{\delta_{E(H)}(e)}{\delta_{V(H)}(u)}\frac{1}{|e|^2}$. For the general adjacency, $A_{(H,\delta_{V(H)},\delta_{E(H)})}=[a_{uv}]_{u,v\in V(H)}$ is such that the diagonal entries are $a_{uu}=0$ for all $u\in V(H)$, and for any two distinct $u,v\in V(H)$, the non-diagonal entry $a_{uv}=q_{uv}$. The matrix representation of the general Laplacian $L_{(H,\delta_{V(H)},\delta_{E(H)})}=[l_{uv}]_{u,v\in V(H)}$ is such that for any $u\in V(H)$, the diagonal position $l_{uu}=\sum\limits_{v(\ne u)\in V(H)}q_{uv}$, and $l_{uv}=-q_{uv}$ for two distinct $u,v\in V(H)$.

\section{Equivalence Relation-Compatible operators associated with Hypergraphs}\label{sec-equi-rel}
Given a graph $G$, a pair of distinct vertices $u,v$ are called \emph{duplicate vertices} if they are non-adjacent and have the same set of neighbours. Consider an equivalence relation $\mathfrak{R}_d$ on the vertex set of the graph such that $(u,v)\in \mathfrak{R}_d$ if and only if $ u,v$ is a pair of duplicate vertices. Since each $\mathfrak{R}_d$-equivalence class $W$ corresponds to $|W|$ number of equal columns in the adjacency matrix, if $|W|>1$, the equivalence class $W$ leads to an eigenvalue $0$ of multiplicity $|W|-1$. Furthermore, the partition of the vertex set due to the equivalence relation $\mathfrak{R}_d$ is an equitable partition. This leads to a quotient matrix of the adjacency matrix of the graph. Besides the graph adjacency matrix, similar spectral behaviour is also observed for the Laplacian matrix, signless Laplacian matrix, and other matrices associated with the graph (see \cite{mehatari-banerjee-motif2015,Faria-vector,Benjamin-webb-equitable-2016}, and references therein). Now, in this section, we show that the symmetry encoded in an equivalence relation on the vertex set of a hypergraph is reflected in the spectra of a specific family of matrices associated with the hypergraph. 

\Lh 
For any equivalence relation $\mf{R}$ on $V(H)$, we denote the collection of all $\mf{R}$-equivalence classes as $\mf{C}_{\mf{R}}$. 
\begin{df}[Equivalence relation-compatible linear operator]\label{df_eq_cmpt_op}
\Lh
For any equivalence relation $\mf{R}$ on $V(H)$, a linear map $T:\mathbb{C}^{V(H)}\to \mathbb{C}^{V(H)}$ is called \emph{$\mf{R} $-compatible} if for all $W\in \mf{C}_{\mf{R}}$, and for any pair  $\{u,v\}\subseteq W$,
$$ (Tx)\circ \phi_{uv}=T (x\circ\phi_{uv}),$$
where $x\in\mathbb{C}^{V(H)}$, and $\phi_{uv}:V(H)\to V(H)$ is defined by 
$$\phi_{uv}(w)=
\begin{cases}
v&\text{~if~}w=u,\\
u&\text{~if~}w=v,\\
w&\text{~otherwise.~}\\
\end{cases}
$$
\end{df}
For any $W\in \mf{C}_{\mf{R}}$, each two-element subset  $\{u,v\}\subseteq W$ leads us to the map $S_{uv}:\mathbb{C}^{V(H)}\to\mathbb{C}^{V(H)}$ which is defined by $S_{uv}(x)=x\circ\phi_{uv}$. Thus, for any equivalence relation $\mf{R}$ on $V(H)$, a linear map $T:\mathbb{C}^{V(H)}\to \mathbb{C}^{V(H)}$ is $\mf{R} $-compatible if $ T\circ S_{uv}=S_{uv}\circ T$ for all $\{u,v\}\subseteq W$. That is, the diagram in \Cref{fig:equiv-comp} commutes for all $\{u,v\}\subseteq W$, and $W\in \mf{C}_{\mf{R}}$.
\begin{figure}[ht]
    \centering
  \begin{tikzcd}[column sep=large]
\mathbb{C}^{V(H)} \ar[r,"T"] \ar[d,"S_{uv}"] & \mathbb{C}^{V(H)} \ar[d,"S_{uv}"] \\
\mathbb{C}^{V(H)} \ar[r,"T"] & \mathbb{C}^{V(H)}
\end{tikzcd}
    \caption{$ T\circ S_{uv}=S_{uv}\circ T.$}
    \label{fig:equiv-comp}
\end{figure}
For example, consider the hypergraph $H$ with the vertex set $V(H)=\{1,2,3,4\}$, and the hyperedge set $E(H)=\{e_1=\{1,2,3\},e_2=\{1,3,4\},e_3=\{1,4,2\}\}$. The relation $\mathfrak{R}=\{(1,1),(2,2),(3,3),(4,4),(2,3),(3,2),(3,4),(4,3),(2,4),(4,2)\}$ is an equivalence relation on $V(H)$ with the collection of equivalence classes $\mf{C}_{\mf{R}}=\{\{1\},\{2,3,4\}\}$. The matrix $L_{(R,H)}=\left[\begin{smallmatrix}
  \phantom{-} 6 & -2 & -2 & -2\\ -2 &  \phantom{-} 4 & -1 & -1\\ -2 & -1 &  \phantom{-} 4 & -1\\ -2 & -1 & -1 &  \phantom{-} 4
\end{smallmatrix}\right].$ Therefore, for any $x:V(H)\to\mathbb{C}$, it holds that $(L_{(R,H)}x)(2)=-2x(1)+4x(2)-x(3)-x(4) $, and $(L_{(R,H)}x)(3)=-2x(1)+4x(3)-x(2)-x(4) $. Consequently, the interchange of $x(2)$, and $x(3)$ leads to the interchange of $(L_{(R,H)}x)(2) $, and $(L_{(R,H)}x)(3)$. That is, $ L_{(R,H)}\circ S_{23}=S_{32}\circ L_{(R,H)}$.

 Given two equivalence relation $\mathfrak{R}_1$, and $\mathfrak{R}_2$ on $V(H)$, the equivalence relation $\mathfrak{R}_1$ is called \emph{finer} than $\mathfrak{R}_2$ if for each $\mathfrak{R}_1$-equivalence class $W$, there exists an $\mathfrak{R}_2$ equivalence class $W'$ such that $W\subseteq W'$. For example, given a hypergraph $H$, suppose that two relations $\mathfrak{R}_1$, and $\mathfrak{R}_2$ on $V(H)$ are defined by $(u,v)\in\mathfrak{R}_1$ if and only if $E_u(H)=E_v(H)$, and $ (u,v)\in\mathfrak{R}_2$ if and only if $u$, and $v$ have the same degree, that is both $u$, and $v$ are incident with the same number of hyperedges in $H$. Both the relations $\mathfrak{R}_1$, and $\mathfrak{R}_2$ are equivalence relations on $V(H)$. Moreover, since $E_u(H)=E_v(H)$ means $u$, and $v$ have the same degree, $\mathfrak{R}_1$ is finer than $\mathfrak{R}_2$.
Given any linear operator $T$ associated with a hypergraph $H$, the following proposition provides an equivalence relation $\mathfrak{R}_T$ such that $T$ is $\mathfrak{R}_T$-compatible. Moreover, we show that if any equivalence relation $\mathfrak{R}$ is finer than $\mathfrak{R}_T$, then $T$ is  $\mathfrak{R}$-compatible.
 \begin{prop}\label{equivalence-exm-prop}
     Let $T:\mathbb{C}^{V(H)}\to\mathbb{C}^{V(H)}$ be a linear operator with matrix representation
     $T=[t_{uv}]_{u,v\in V(H)}$. The relation $\mathfrak{R}_T=\{(u,v)\in V(H)\times V(H):t_{uu}=t_{vv}, t_{uv}=t_{vu}, \text{~and~}t_{uw}=t_{vw},t_{wu}=t_{wv}\text{~for~}w\in V(H)\setminus\{u,v\}\}$ is an equivalence relation and $T$ is $\mathfrak{R}_T$-compatible linear operator. For any equivalence relation $\mathfrak{R}$ on $V(H)$, if $\mathfrak{R}$ is finer than $\mathfrak{R}_T$, then $T$ is $\mathfrak{R}$-compatible linear operator.
 \end{prop}
 \begin{proof}
      If $W\in \mf{C}_{\mf{R}}$ contains at least $2$ elements and $u,v(\in W)$ are two distinct vertices, then for any $w\in V(H)\setminus \{u,v\}$ we have $t_{wu}=t_{wv}$. Thus, for any $x\in \mathbb{C}^{V(H)}$, if $w\notin\{u,v\}$, then $t_{wu}x(u)+t_{wv}x(v)=t_{wv}(x\circ \phi_{uv})(v)+t_{wu}(x\circ \phi_{uv})(u)$ and therefore,
     \begin{align*}
    & ((Tx)\circ\phi_{uv})(w)=  (Tx)(w)=\sum_{w'\in V(H)}t_{ww'}x(w')=\sum_{w'\in V(H)}t_{ww'}(x\circ \phi_{uv})(w')=(T(x\circ\phi_{uv}))(w).
     \end{align*}
    Now, $t_{uv}=t_{vu}$, and $t_{uu}=t_{vv}$, lead us to $ t_{uu}x(u)+t_{uv}x(v)=t_{vv}x(u)+t_{vu}x(v)=t_{vv}(\phi_{uv}\circ x)(v)+t_{vu}(\phi_{uv}\circ x)(u)$. 
    Therefore, $t_{uw}=t_{vw}$ for all $w\notin \{u,v\}$ implies that
    \begin{align*}
        &((Tx)\circ\phi_{uv})(v)=(Tx)(u)=\sum_{w\in V(H)}t_{uw}x(w)=\sum_{w\in V(H)}t_{vw}(x\circ\phi_{uv})(w)=(T(x\circ\phi_{uv}))(v).
    \end{align*}
    Similarly we can show that $ ((Tx)\circ\phi_{uv})(u)=(T(x\circ\phi_{uv}))(u)$. Therefore, $T$ is $\mathfrak{R}_T$-compatible.
    For any equivalence relation $\mathfrak{R}$ on $V(H)$ which is finer than $\mathfrak{R}_T$, if $W$ is an  $\mathfrak{R}$-equivalence class then $W\subseteq W'$ where $W'$ is an $\mathfrak{R}_T$-equivalence class. Therefore, for all $u,v\in W\subseteq W'$, we have $(Tx)\circ\phi_{uv}=T(x\circ\phi_{uv})$. That is $T$ is $\mathfrak{R}$-compatible.
 \end{proof}
 We demonstrate the use of the \Cref{equivalence-exm-prop} in the following example.
 \begin{exm}\label{equivalence-comp-exm}\rm Let $H$ be a hypergraph with
 $V(H)=\{u_1,v_1,w_1,u_2,v_2,w_2\}$, $E(H)=\{e_1=\{u_1,u_2,v_2,w_2\},$ $e_2=\{v_1,u_2,v_2,w_2\},$$e_3=\{w_1,u_2,v_2,w_2\},$$e_4=\{u_1,v_1,w_1\}\}$. Consider the partition $V(H)=V_1\cup V_2$, where $V_1=\{u_1,v_1,w_1\}$, and $V_2=\{u_2,v_2,w_2\}$. Let $\mathfrak{R}=\{(u,v)\in V_i\times V_i, i=1,2\}$ be the equivalence relation induced by the partition $V(H)=V_1\cup V_2$. For the hypergraph matrix $A_{(R,H)}=[a^c_{uv}]_{u,v\in V(H)}$, since the diagonal entries of $A_{(R,H)}$ are $0$, and for $u,v(\ne u)\in V(H)$, the $(u,v)$-th entry is $a_{uv}=|E_u(H)\cap E_v(H)|$.  Now for  $u\in V_1$, if $v(\ne u)\in V_1$, then $|E_u(H)\cap E_v(H)|=1$, otherwise if   $v(\ne u)\in V_2$, then $|E_u(H)\cap E_v(H)|=1$. For $u\in V_2$, if $v(\ne u)\in V_1$, then $|E_u(H)\cap E_v(H)|=1$, otherwise if   $v(\ne u)\in V_2$, then $|E_u(H)\cap E_v(H)|=3$. Therefore, each of the $\mathfrak{R}$-equivalence class (that is, each of $V_1$, and $V_2$) lie inside an $\mathfrak{R}_{A_{(R,H)}}$-equivalence class. That is $\mathfrak{R}$ is finer than $\mathfrak{R}_{A_{(R,H)}}$. Therefore, by the \Cref{equivalence-exm-prop}, $A_{(R,H)}$ is  $\mathfrak{R}$-compatible. Similarly, we can show the same fact for the matrices $Q_{(R,H)}$ and $L_{(R,H)}$.
 \end{exm}
 In general, an equivalence $\mathfrak{R}$ on $V(H)$ does not necessarily have a relation to the hyperedges. But if a hypergraph operator $T$ is related to the hyperedges, and $T$ is compatible with $\mathfrak{R}$, then this compatibility means $\mathfrak{R}$ has some relations with the hyperedges. For example, consider the adjacency matrix $A_{(R,H)}=[a_{uv}]_{u,v\in V(H)}$, where for distinct $u,v\in V(H)$, the non-diagonal entry $a_{uv}=|E_u(H)\cap E_v(H)|$, the number of hyperedges that contain both $u,v$ and all the diagonal entries are $0$ (cf.\cite{bretto2013hypergraph}). If $A_{(R,H)}$ is compatible with an equivalence relation $\mathfrak{R}$ on $V(H)$, then $(A_{(R,H)}x)\circ \phi_{uv}=A_{(R,H)}(x\circ\phi_{uv})$ leads us to $|E_u(H)\cap E_w(H)|=|E_v(H)\cap E_w(H)|$ for all $W\in \mf{C}_{\mf{R}}$, $\{u,v\}\subseteq W$, and $w\notin\{u,v\}$. In the upcoming sections, we are going to show some equivalence relations such that specific hypergraph matrices are compatible with these equivalence relations. The \Cref{equivalence-exm-prop} will be used to show these matrices are compatible with specific equivalence relations. Now, we prove some results related to the spectra of equivalence relation compatible operators.
 The existence of a set of duplicate vertices \( W \) in a graph $G$ with $|W|>1$ leads to a $0 $ eigenvalue of the adjacency matrix of $G$.
 Similarly, here we find that for a hypergraph \( H \), if \( \mathfrak{R} \) is an equivalence relation on the vertex set \( V(H) \), then any \( \mathfrak{R} \)-equivalence class \( W \) with \( |W| > 1 \) corresponds to an eigenvalue of every \( \mathfrak{R} \)-compatible operator associated with \( H \).
\begin{thm}[Equivalence class-eigenvalue theorem ]\label{equivalence-eig}
\Lh
If $\mf{R}$ is an equivalence relation on $V(H)$ and $T$ is an $\mf{R} $-compatible linear operator, then for any $W\in \mf{C}_{\mf{R}}$ with $|W|>1$,
$(T(x_{uv}))(v)$ is an eigenvalue of $T$ with an eigenvector $x_{uv}$, where distinct vertices $ u,v\in W$.
\end{thm}
\begin{proof}
Since $T$ is a $\mf{R} $-compatible linear operator, 
\begin{align*}
    (Tx_{uv})\circ \phi_{uv}&=T ( x_{uv}\circ \phi_{uv})= (Tx_{vu})=-(Tx_{uv})
\end{align*}
Thus for all $w\notin\{u,v\}$, 
\begin{align*}
         &((Tx_{uv})\circ \phi_{uv})(w)=-(Tx_{uv})(w)\\
         &\implies ((Tx_{uv})(w)=-(Tx_{uv})(w)
\end{align*}
Thus, $(Tx_{uv})(w)=0$, for all $w\notin\{u,v\}$.
Moreover,
\begin{align*}
    &((Tx_{uv})\circ \phi_{uv})(u)=-(Tx_{uv})(u)\\
     &\implies ((Tx_{uv}) (v)=-(Tx_{uv})(u).
\end{align*}
Therefore, $Tx_{uv}=[(Tx_{uv})(v)](x_{uv} )$.
\end{proof}
Since $\mf{C}_{\mf{R}}$ is the collection of all the $\mf{R}$-equivalence classes, \Cref{equivalence-eig} provides $|V(H)|-|\mf{C}_{\mf{R}}|$ number of eigenvalues of any $\mf{R} $-compatible operator. 

\begin{lem}\label{const-lem}
\Lh
If $\mf{R}$ is an equivalence relation on $V(H)$ and, $W\in\mf{C}_{\mf{R}}$ with $|W|>1$. For any $\mf{R} $-compatible linear operator  $T$, and any three distinct vertices $u,u^\prime,v \in W$, 
$(T(x_{uv}))(v)=(T(x_{u^\prime v}))(v)$.
\end{lem}
\begin{proof}
    Since $T$ is a $\mf{R} $-compatible linear operator, $ (Tx)\circ \phi_{uv}=T (x\circ\phi_{uv})$ for all $x\in\mathbb{C}^{V(H)}$, and $u,v\in W$.
    Therefore, 
    \begin{align*}
       & ( (Tx_{uv})\circ \phi_{uu^\prime})(v)=(T (x_{uv}\circ\phi_{uu^\prime}))(v)\\
         & \implies ( (Tx_{uv})(v)=(T (x_{u^\prime v}))(v).
       \end{align*}
       This completes the proof.
\end{proof}
Thus the previous result suggests that for an equivalence relation $\mf{R}$ on $V(H)$ and an $\mf{R} $-compatible linear operator  $T$, each $W\in\mf{C}_{\mf{R}}$ with $|W|>1$ corresponds to a constant, $\lambda_{(W,T)}=(Tx_{uv})(v)$, 
for all $u\in W\setminus\{v\}$.
Since $\{x_{uv}:u\in W\setminus\{v\}\}$ is a basis of $S_W$, we have the following result.
\begin{thm}[Equivalence class-eigenspace Theorem]\label{Tw}
    \Lh
If $\mf{R}$ is an equivalence relation on $V(H)$, $T$ is an $\mf{R} $-compatible linear operator, 
 and $W\in\mf{C}_{\mf{R}}$ with $|W|>1$, then $\lambda_{(W,T)}$ is an eigenvalue of $T$ with the multiplicity at least $|W|-1$, and $S_W$ is a subspace of the eigenspace of $\lambda_{(W,T)}$.
\end{thm}
\begin{proof}
    Let $W=\{v_0,\ldots, v_k\}$. By \Cref{equivalence-eig}, and \Cref{const-lem}, we have $\lambda_{(W,T)}$ is an eigenvalue of $T$. Since for $i=1,\ldots,k$,  $x_{v_0v_i}$ is an eigenvector corresponding to the eigenvalue $\lambda_{(W,T)}$, $S_W$ is a subspace of the eigenspace of $\lambda_{(W,T)}$.
\end{proof}
From now onwards, we refer to the constant $\lambda_{(W,T)}$ as the \emph{eigenvalue of $T$ corresponding to the equivalence class $W$}.
\begin{exm}
     Consider the hypergraph $H$ and the equivalence relation $\mathfrak{R}$ described in the \Cref{equivalence-comp-exm}. That example shows that $A_{(R,H)}$ is an $\mathfrak{R}$-compatible operator. By the \Cref{Tw}, the  $\mathfrak{R}$-equivalence classes $V_1$, and $V_2$ corresponds to the eigenvalue $\lambda_{(V_1,A_{(R,H)})}=-1 $, and $\lambda_{(V_2,A_{(R,H)})}=-3 $ with multiplicity at least $2$ for each.
\end{exm}
Using \Cref{equivalence-eig}, and \Cref{Tw}, we have  $|V(H)|-|\mf{C}_{\mf{R}}|$  number of eigenvalues of an $\mf{R}$-compatible linear operator $T$. For the remaining eigenvalues, we have the following result.

\begin{lem}\label{lem_char}
        Let $H$ be a hypergraph, $\mf{R}$ be an equivalence relation on $V(H)$, and $T$ be an $\mf{R} $-compatible linear operator. If $\mf{C}_\mf{R}=\{W_{1},\ldots, W_{m}\}$, then there exists a matrix 
$[c^i_j]_{i,j=1}^m$ such that $T\chi_{_{W_i}}=\sum\limits_{j=1}^mc^i_j\chi_{_{W_j}}$.
\end{lem}
\begin{proof}
    Since $T$ is $\mf{R} $-compatible, for all $W_i,W_j\in \mf{C}_{\mf{R}}$, and $u,v\in W_{i}$, 
    \begin{equation}\label{lem-main}
   (T\chi_{_{W_j}})\circ \phi_{uv}=T (\chi_{_{W_j}}\circ\phi_{uv}).
   \end{equation}
    Since $\chi_{_{W_j}}\circ\phi_{uv}=\chi_{_{W_j}}$, by \Cref{lem-main} we have  $ (T\chi_{_{W_j}})\circ \phi_{uv}=T(\chi_{_{W_j}})$. Thus, $(T(\chi_{_{W_j}}))(u)=(T(\chi_{_{W_j}}))(v)$ for all $u,v\in W_i$. That is, $(T(\chi_{_{W_j}}))(u)=c^j_i$, a constant for all $u\in W_i$.  Therefore, there exists a real $c^i_j$, for all $i,j=1,\ldots,m$, such that, 
    $T\chi_{_{W_i}}=\sum\limits_{j=1}^mc^i_j\chi_{_{W_j}}.$
    This completes the proof.
\end{proof}
Suppose that $\mathbb{C}^{\mf{C}_{\mf{R}}}$ is the vector space of all complex-valued functions on $\mf{C}_{\mf{R}}$  of the form $y:\mf{C}_{\mf{R}}\to \mathbb{C}$. If $\epsilon_{W_i}:\mf{C}_{\mf{R}}\to \{0,1\}$ is defined by $\epsilon_{W_i}(W_j)=1$ if $i=j$, and otherwise $\epsilon_{W_i}(W_j)=0$,
then $\{\epsilon_{W_i}:i=1,\ldots,m\}$ is a basis of $\mathbb{C}^{\mf{C}_{\mf{R}}}$.
Upon identifying each $\mf{R}$-equivalence class as a single element,  for any $\mf{R} $-compatible linear operator $T$, and the above Lemma provides us a linear operator
$ T_{\mf{R}}:\mathbb{C}^{\mf{C}_{\mf{R}}}\to \mathbb{C}^{\mf{C}_{\mf{R}}}$
such that, 
$T_{\mf{R}}(\epsilon_{W_i})=\sum\limits_{j=1}^mc^i_j\epsilon_{W_j}$ for all $i=1,\ldots,m$. Thus, for any $y:\mf{C}_{\mf{R}}\to \mathbb{C}$, 
\begin{equation}\label{R-contract_exp}
    T_{\mf{R}}(y)=\sum\limits_{i,j=1}^mc^i_jy(W_i)\epsilon_{W_j}.
\end{equation}
We refer to $T_{\mf{R}}$ as the \emph{${\mf{R}}$-contraction} of $T$. Thus, with respect to the ordered basis, $\{\epsilon_{W_1},\ldots,\epsilon_{W_n}\} $ the matrix representation of $T_{\mathfrak{R}}$ is 
\[\left[
\begin{smallmatrix}
c^1_1&c^2_1&\cdots&c^n_1\\
\vdots&\vdots&\ddots&\vdots\\
c^1_n&c^2_n&\cdots&c^n_n
\end{smallmatrix}
\right].\]
Unless some basis is mentioned, we only consider this matrix representation as the matrix of $T_{\mathfrak{R}}$.
For $y:\mf{C}_{\mf{R}}\to \mathbb{C}$, we define the \emph{$\mf{R}$-blow up} of $y$ is a function $y_{_\mf{R}}:V(H)\to \mathbb{C}$, and is defined by $ y_{_\mf{R}}(v)=y({W_i})$, for all $v\in W_i$, for all $i=1,\ldots,m$. Thus, $y_{_\mf{R}}=\sum\limits_{i=1}^my(W_i)\chi_{_{W_i}}$.
\begin{exm}\rm\label{ex-contraction-lem}
    Again, we demonstrate the \Cref{lem_char}, using the hypergraph $H$, and the equivalence relation $\mathfrak{R}$ on $V(H)$ we have considered in the \Cref{equivalence-comp-exm}. Since the matrix $A_{(R,H)}=[a_{uv}]_{u,v\in V(H)}$ is $\mathfrak{R}$-compatible, for all $i=1,2$, if distinct $u,v\in V_i$, then $a_{uw}=a_{vw}$ for all $w\in V(H)\setminus\{u,v\}$, and $a_{uv}=|E_u(H)\cap E_v(H)|=a_{vu}$. Therefore, $(A_{(R,H)}\chi_{V_i})(u)=\sum_{w\in V_i}a_{uw}=\sum_{w\in V_i}a_{vw}=(A_{(R,H)}\chi_{V_i})(v)=c^i_j$ for all $u,v\in V_j$ for all $i,j=1,2$. That is, $(A_{(R,H)}\chi_{V_i})(u)=c^i_j $ for all $u\in V_j$, for $i,j=1,2$. Therefore, the matrix representation of $\mathfrak{R}$-contraction of $A_{(R,H)}$ is $\left[\begin{smallmatrix}
        c^1_1&c_1^2\\
        c^1_2&c^2_2
    \end{smallmatrix}\right]=\left[\begin{smallmatrix}
       2&3\\
        3&6
    \end{smallmatrix}\right]$. 
\end{exm}
In the theory of equitable partition in the adjacency matrix associated with a graph, it is well-established that each eigenvalue of the quotient matrix is an eigenvalue of the original matrix (see \cite[Theorem 9.3.3]{agt-godsil-royle}). Now, here we show that given any equivalence relation ${\mf{R}}$ on the vertex set of a hypergraph, a similar fact holds for the $\mathfrak{R}$-contraction $T_{\mf{R}}$ of any ${\mf{R}}$-compatible operator $T$.
\begin{thm}\label{gen-contract-eig}
    Let $H$ be a hypergraph, $\mf{R}$ be an equivalence relation on $V(H)$, and $T$ be a $\mf{R} $-compatible linear operator. If $\lambda$ is an eigenvalue of $T_{\mf{R}}$ with eigenvector $y$, then $\lambda$ is an eigenvalue of $T$ with an eigenvector $y_{_\mf{R}}$.
\end{thm}
\begin{proof}
    For any $v\in V(H)$,
    \begin{align*}
       ( Ty_{_\mf{R}})&=T\left(\sum\limits_{i=1}^my(W_i)\chi_{_{W_i}}\right)=\sum\limits_{i=1}^my(W_i)T(\chi_{_{W_i}})=\sum\limits_{i=1}^my(W_i)\sum\limits_{j=1}^mc^i_j\chi_{_{W_j}}.      
    \end{align*}
    Now, for any $j=1,\ldots,m$, if $v\in W_j$, then $( Ty_{_\mf{R}})(v)=\sum\limits_{i=1}^mc^i_jy(W_i)$.
So from \Cref{R-contract_exp} we have
$(T_{\mf{R}}(y))(W_j)=\sum\limits_{i=1}^my(W_i)c^i_j$. Since for any $v\in V(H)$, there exists a $j\in\{1,\ldots,m\}$ such that $v\in W_j$ we have,  $( Ty_{_\mf{R}})(v)= (T_{\mf{R}}(y))(W_j)=\lambda y(W_i)=\lambda y_{_\mf{R}}(v)$.
    Therefore, $Ty_{_\mf{R}}=\lambda y_{_\mf{R}}$, and this completes the proof.
\end{proof}
For exampe, since the eigenvalues of the matrix $\left[\begin{smallmatrix}
        c^1_1&c_1^2\\
        c^1_2&c^2_2
    \end{smallmatrix}\right]=\left[\begin{smallmatrix}
       2&3\\
        3&6
    \end{smallmatrix}\right]$ in the \Cref{ex-contraction-lem} are $4\pm\sqrt{13}$, two eigenvalues of $A_{(R,H)}$ are $4\pm\sqrt{13}$.

\section{Building blocks in hypergraphs}
Now, we study some particular equivalence relations on the vertex set of a hypergraph such that all the hyperedges of the hypergraph can be expressed in terms of the equivalence classes.  These equivalence classes are crucial because they help us to understand how the hypergraph is structured. Since the incidence structure of the hypergraph is related to these equivalence classes, we refer to them as \emph{building blocks} of a hypergraph.  
\subsection{Units in a hypergraph}\label{1.unit}
In a hypergraph, a collection of vertices with the same star has similar incidence relations. This fact motivates us to define the following equivalence relation.
\begin{df}[Unit]\label{df_unit}
\Lh
The relation
$\mf{R}_{s}=\{(u,v)\in V(H)\times V(H):E_u(H)=E_v(H)\}$
is an equivalence relation. We refer to each $\mf{R}_{s}$-equivalence class in a hypergraph as a \emph{unit}. 
\end{df}
Given a hypergraph $H$, we denote the collection of all the units in $H$ as $\mathfrak{U}(H)$.
A unit $W_E$ in $H$ corresponds to an $E\subseteq E(H)$, such that 
$W_E=\{v\in V(H):E_v(H)=E\}$.
 We refer to $E$ as the \emph{generating set of the unit $W_E$.}
 \begin{exm}
\begin{figure}[H]
		\centering
		\begin{tikzpicture}[scale=0.45]
				\node [style=none] (0) at (-1.75, -0.25) {};
				\node [style=none] (1) at (0, 1) {};
				\node [style=none] (2) at (1, 0) {};
				\node [style=none] (3) at (-0.5, -0.75) {};
				\node [style=none] (4) at (5.25, 0.25) {};
				\node [style=none] (5) at (7, 2) {};
				\node [style=none] (6) at (8.5, -0.25) {};
				\node [style=none] (7) at (7, -1) {};
				\node [style=none] (8) at (2.75, 5.25) {};
				\node [style=none] (9) at (3.75, 6.5) {};
				\node [style=none] (10) at (6.5, 5.25) {};
				\node [style=none] (11) at (5.25, 4.5) {};
				\node [style=none] (12) at (2.75, -5) {};
				\node [style=none] (13) at (5, -4) {};
				\node [style=none] (14) at (7, -5.75) {};
				\node [style=none] (15) at (5, -7) {};
				\node [style=none] (16) at (-6.5, 4.25) {};
				\node [style=none] (17) at (-5.5, 5.25) {};
				\node [style=none] (18) at (-2, 3) {};
				\node [style=none] (19) at (-5.5, 2.5) {};
				\node [style=none] (20) at (-7.25, -4.5) {};
				\node [style=none] (21) at (-3.25, -3.25) {};
				\node [style=none] (22) at (-4, -5.25) {};
				\node [style=none] (23) at (-7, -6.5) {};
				\node [style=none] (24) at (2, 0.5) {};
				\node [style=none] (25) at (1.5, 2.25) {};
				\node [style=none] (26) at (3, -0.5) {};
				\node [style=none] (27) at (1.25, -1.75) {};
				\node [style=none] (28) at (-2.25, -0.5) {};
				\node [style=none] (29) at (4.25, 7) {};
				\node [style=none] (30) at (7.25, 4.5) {};
				\node [style=none] (31) at (3.25, 0.75) {};
				\node [style=none] (32) at (1.25, -2.75) {};
				\node [style=none] (33) at (8.25, 2.25) {};
				\node [style=none] (34) at (7.5, -1.5) {};
				\node [style=none] (35) at (4.5, -1.25) {};
				\node [style=none] (36) at (5.75, 3) {};
				\node [style=none] (37) at (5.25, -2.5) {};
				\node [style=none] (38) at (8, -5.5) {};
				\node [style=none] (39) at (5, -7.75) {};
				\node [style=none] (40) at (2.25, 2.25) {};
				\node [style=none] (41) at (-4.25, -1) {};
				\node [style=none] (42) at (1.5, 0.75) {};
				\node [style=none] (43) at (-2, -5) {};
				\node [style=none] (44) at (-7.75, -3.75) {};
				\node [style=none] (45) at (-7.25, -7.25) {};
				\node [style=none] (46) at (-2, -5) {};
				\node [style=none] (47) at (-3, -0.25) {};
				\node [style=none] (48) at (1.75, 0.25) {};
				\node [style=none] (49) at (-6.5, 2.5) {};
				\node [style=none] (50) at (-6, 6) {};
				\node [style=none] (51) at (-0.75, 3.75) {};
				\node [style=none] (52) at (1.25, 1.5) {};
				
				\node [style=none] (79) at (-6.5, 0) {};
				\node [style=none] (80) at (-9, 1) {};
				\node [style=none] (81) at (-11, 0.25) {};
				\node [style=none] (82) at (-8.5, -0.75) {};
				\node [style=none] (83) at (-4.75, -1) {};
				\node [style=none] (84) at (-12.25, 0.75) {};
				\node [style=none] (85) at (-2.25, -3.5) {};
				\node [style=none] (86) at (-4.75, -7.5) {};
				\node [style=none] (87) at (-11.25, -2.75) {};
				\node [style=none] (88) at (-6, -0.25) {};
				\node [style=none] (89) at (-11.5, 0.25) {};
				\node [style=none] (90) at (-2, 2.5) {};
				\node [style=none] (91) at (-3, 6) {};
				\node [style=none] (92) at (-8.75, 3.75) {};
				\node [style=none] (93) at (-10.25, 2.5) {};
				
				\draw [style=new edge style 1, bend left=315,fill=gray!50!white] (1.center) to (0.center)
				to(0.center) to (3.center)
				to(3.center) to (2.center);
				to (2.center) to (1.center);
				
				\draw [style=new edge style 1, bend left=315,fill=gray!50!white] (5.center) to (4.center)to
				(4.center) to (7.center)to
				(7.center) to (6.center)to (6.center) to (5.center);
				
				\draw [style=new edge style 1, bend left=315,fill=gray!50!white] (9.center) to (8.center)to
				(8.center) to (11.center)
				to (11.center) to (10.center)
				to(10.center) to (9.center);
				
				\draw [style=new edge style 1, bend left=315,fill=gray!50!white] (13.center) to (12.center)
				to (12.center) to (15.center)
				to(15.center) to (14.center)
				to (14.center) to (13.center);
				
				\draw [style=new edge style 1, bend left=315,fill=gray!50!white] (17.center) to (16.center)
				to(16.center) to (19.center)
				to(19.center) to (18.center)
				to(18.center) to (17.center);
				
				\draw [style=new edge style 1, bend left=315,fill=gray!50!white] (21.center) to (20.center)
				to(20.center) to (23.center)
				to(23.center) to (22.center)
				to(22.center) to (21.center);
				
				\draw [style=new edge style 1, bend left=315 ,fill=gray!50!white] (25.center) to (24.center)
				to(24.center) to (27.center)
				to(27.center) to (26.center)
				to(26.center) to (25.center);
				\draw [bend left=60, looseness=0.50] (28.center) to (29.center);
				\draw [bend left=75, looseness=0.75] (29.center) to (30.center);
				\draw [bend right=45, looseness=0.50] (30.center) to (31.center);
				\draw [bend left=45, looseness=0.75] (31.center) to (32.center);
				\draw [bend left=75, looseness=0.50] (32.center) to (28.center);
				\draw [bend left=75] (33.center) to (34.center);
				\draw [bend left, looseness=0.50] (34.center) to (35.center);
				\draw [bend left=60, looseness=0.75] (32.center) to (35.center);
				\draw [bend left=45] (28.center) to (36.center);
				\draw [bend left=15, looseness=0.50] (36.center) to (33.center);
				\draw [in=90, out=-60, looseness=1.50] (37.center) to (38.center);
				\draw [bend left] (38.center) to (39.center);
				\draw [bend left, looseness=0.50] (39.center) to (28.center);
				\draw [bend left=60] (28.center) to (40.center);
				\draw [bend left=15, looseness=0.50] (40.center) to (37.center);
				\draw [bend left=45] (41.center) to (42.center);
				\draw [bend left, looseness=0.75] (42.center) to (43.center);
				\draw [bend right=75, looseness=0.75] (44.center) to (45.center);
				\draw [bend right=15, looseness=0.75] (45.center) to (46.center);
				\draw [bend right=15, looseness=0.50] (44.center) to (41.center);
				\draw [bend right=45] (47.center) to (48.center);
				\draw [bend left=75, looseness=0.75] (49.center) to (50.center);
				\draw [bend left=15, looseness=0.75] (50.center) to (51.center);
				\draw [bend left=15, looseness=0.50] (49.center) to (47.center);
				\draw [bend left=15, looseness=0.75] (51.center) to (52.center);
				\draw [bend left=15] (52.center) to (48.center);
				\draw [style=new edge style 1, bend right=315,fill=gray!50!white] (80.center) to (79.center)
				to (79.center) to (82.center)
				to (82.center) to (81.center)
				to (81.center) to (80.center);
				
				\draw [bend right=45] (83.center) to (84.center);
				\draw [bend left=75, looseness=0.75] (85.center) to (86.center);
				\draw [bend left=15, looseness=0.75] (86.center) to (87.center);
				\draw [bend left=15, looseness=0.50] (85.center) to (83.center);
				\draw [bend left=45] (88.center) to (89.center);
				\draw [bend right=75, looseness=0.75] (90.center) to (91.center);
				\draw [bend right, looseness=0.75] (91.center) to (92.center);
				\draw [bend right=15, looseness=0.50] (90.center) to (88.center);
				\draw [bend right=15, looseness=0.75] (92.center) to (93.center);
				\draw [bend right=15] (93.center) to (89.center);
				\draw [bend right=15, looseness=0.75] (84.center) to (87.center);
				\node [style=new style 0] (53) at (-1.25, 0) {\tiny 1};
				\node [style=new style 0] (54) at (0.5, 0.25) {\tiny 2};
				\node [style=new style 0] (55) at (2.25, 1) {\tiny 4};
				\node [style=new style 0] (56) at (2, -0.75) {\tiny 3};
				\node [style=new style 0] (57) at (3.25, 5.5) {\tiny 10};
				\node [style=new style 0] (58) at (5.5, 6) {\tiny 9};
				\node [style=new style 0] (59) at (6.25, 1.25) {\tiny 8};
				\node [style=new style 0] (60) at (7.75, -0.5) {\tiny 7};
				\node [style=new style 0] (61) at (6.25, -6.25) {\tiny 6};
				\node [style=new style 0] (61) at (5.5, -5) {\tiny 15};
				\node [style=new style 0] (62) at (3.25, -5) {\tiny 5};
				\node [style=new style 0] (63) at (-4, -4.25) {\tiny 13};
				\node [style=new style 0] (64) at (-7, -5.75) {\tiny14};
				\node [style=new style 0] (65) at (-6, 3.25) {\tiny 12};
				\node [style=new style 0] (66) at (-3, 3.5) {\tiny11};
				\node [style=none] (67) at (-0.25, -0.5) {\tiny$W_{E_1}$};
				\node [style=none] (68) at (2.5, 0) {\tiny$W_{E_2}$};
				\node [style=none] (69) at (4.75, -5.75) {\tiny$W_{E_3}$};
				\node [style=none] (70) at (6.5, 0.1) {\tiny$W_{E_4}$};
				\node [style=none] (71) at (4.75, 5.25) {\tiny$W_{E_5}$};
				\node [style=none] (72) at (-4.5, 3) {\tiny$W_{E_6}$};
				\node [style=none] (73) at (-5.75, -5) {\tiny$W_{E_7}$};
				\node [style=none] (74) at (6, 7.25) {$e_1$};
				\node [style=none] (75) at (9.5, 0.25) {$e_2$};
				\node [style=none] (76) at (8, -6.75) {$e_3$};
				\node [style=none] (77) at (-0.75, -4.75) {$e_5$};
				\node [style=none] (77) at (-0.75, -7.75) {$H$};
				\node [style=none] (78) at (-0.8, 4.3) {$e_4$};
				\node [style=new style 0] (94) at (-4.5, 4.5) {\tiny 16};
				\node [style=new style 0] (95) at (-10.25, 0.5) {\tiny 17};
				\node [style=none] (96) at (-9.25, -0.25) {\tiny $W_{E_8}$};
				\node [style=new style 0] (97) at (-7.75, -0.25) {\tiny 18};
				\node [style=none] (98) at (-8, 5) {$e_6$};
				\node [style=none] (99) at (-10.75, -4) {$e_7$};
			\end{tikzpicture}
			 \caption{
    A hypergraph $H$ with $V(H)=\{n\in \n:n\le 18\}$ and 
	    $E(H)=\{e_1,e_2,e_3,e_4,e_5,e_6,e_7\}$, where  $e_1=\{1,2,3,4,9,10\}$, $e_2=\{1,2,3,4,7,8\}$, $e_3=\{1,2,3,4,5,6,15\}$, $e_4=\{1,2,11,12,16\}$, $e_5=\{1,2,13,14\}$, $e_6=\{11,12,16,17,18\}$, and $e_7=\{13,14,17,18\} $ (see \Cref{fig:unit}).
	    The generating sets are $E_1=\{e_1,e_2,e_3,e_4,e_5\}$, $E_2=\{e_1,e_2,e_3\}$, $E_3=\{e_3\}$, $E_4=\{e_2\}$, $E_5=\{e_1\}$, $E_6=\{e_4,e_6\}$, $E_7=\{e_5,e_7\}$ and $E_8=\{e_6,e_7\}$.
	    The units of $H $ are $W_{E_1}=\{1,2\}$, $W_{E_2}=\{3,4\}$, $W_{E_3}=\{5,6,15\}$, $W_{E_4}=\{7,8\}$, $W_{E_5}=\{9,10\}$, $W_{E_6}=\{11,12,16\}$, $W_{E_7}=\{13,14\}$, and $W_{E_8}=\{17,18\}$.
   }
	    \label{fig:unit}
		\end{figure}
	
Consider the hypergraph $H$ in \Cref{fig:unit}. Here $$\mf{U}(H)=\{W_{E_1},W_{E_2}, W_{E_3}, W_{E_4},W_{E_5},W_{E_6},W_{E_7},W_{E_8}\}.$$
If we consider $W_{E_1}=\{1,2\}$, then $E_1(H)=E_2(H)=E_1$ thus, $\{1,2\}\in \mf{R}_s$. Similarly, for any $u,v\in W_{E_6} $, we have $\{u,v\}\in\mf{R}_s$, and $E_u(H)=E_v(H)=E_6$. 
\end{exm}

We refer to a pair of distinct units $W_E$, and $W_F$ as \emph{neighbour units} if there exists $e\in E(H)$ such that $W_E\cup W_F\subseteq e$. For example, in the hypergraph given in \Cref{fig:unit}, $W_{E_6}, W_{E_8}$ are unit neighbours because the hyperedge $e_6$ contains their union, and  $W_{E_6}, W_{E_7}$ are not unit-neighbours because there exists no hyperedge in $H$ to contain $W_{E_6}\cup W_{E_7}$.
The following Proposition show some properties of units.
\begin{prop}\label{edge-unit}
	
    Let $H$ be a hypergraph, and $\mathfrak{U}(H)$ is the collection of all the units in $H$.

    1. For $W_E,W_F\in\mathfrak{U}(H)$, either $W_E=W_F$ or $W_E\cap W_F=\emptyset$, and $V(H)=\bigcup\limits_{W_E\in\mathfrak{U}(H)}W_E$ is a partition of the vertex set $V(H)$.

    2. A hyperedge $e(\in E(H))$ is either a unit or can be expressed as the disjoint union of units.
\end{prop}
\begin{proof}
    Since $W_E,W_F\in\mathfrak{U}(H)$ are two $\mathfrak{R}_s$ equivalence class, either $W_E=W_F$ or $W_E\cap W_F=\varnothing$. Since $\mathfrak{U}(H)$ is the collection of all the equivalence classes of the equivalence relation $\mathfrak{R}_s$ on $V(H)$, we have $V(H)=\bigcup\limits_{W_E\in\mathfrak{U}(H)}W_E$.

    Suppose $e=\{v_1,v_2\ldots,v_k\}\in E(H)$. Now there exists $W_{E_{v_i}(H)}\in \mathfrak{U}(H)$ for $i=1,2,\ldots,k$. Since either $W_{E_{v_i}(H)}=W_{E_{v_j}(H)}$ or $W_{E_{v_i}(H)}\cap W_{E_{v_j}(H)}=\emptyset$ for all $i,j=1,2,\ldots,k$, we just need to show if any unit $W_E$ is such that $W_E\cap e \ne\emptyset$, then $W_E\subseteq e$. Since $W_E\cap e\ne\emptyset$, there exists $u\in W_E\cap e$. Thus $E_u(H)=E$, and $e\in E$. If possible, let $W_E\setminus e\ne\emptyset$, and $v\in W_E\setminus e $. In that case $e\notin E_v(H)=E$, a contradiction. Therefore, $W_E\setminus e=\emptyset$ and this completes the proof.
\end{proof}
We denote the number of units in a hyperedge $e\in E(H)$ as $n_e$. For instance, in the hypergraph $H$ illustrated in the \Cref{fig:unit}, the hyperedges $e_1=W_{E_1}\cup W_{E_2}\cup W_{E_5}$, $e_2=W_{E_1}\cup W_{E_2}\cup W_{E_4}$, $e_3=W_{E_1}\cup W_{E_2}\cup W_{E_3} $, $e_4=W_{E_1}\cup W_{E_6}$, $e_5=W_{E_1}\cup W_{E_7}$, $e_6=W_{E_6}\cup W_{E_8}$, $e_7=W_{E_7}\cup W_{E_8}$. That is, $n_{e_1}=n_{e_2}=n_{e_3}=3$, and $ n_{e_4}=n_{e_5}=n_{e_6}=n_{e_7}=2$.

In any hypergraph $H$, for any hypergraph $e\in E(H)$, the \emph{unit-contraction of the hyperedge} $e$ is defined as $ \hat{e}=\{W_{E_v(H)}:v\in e\}$. Here, $W_{E_v(H)}$ is the unit containing the vertex $v$. For instance, in the hypergraph $H$ illustrated in the \Cref{fig:unit}, $ {\hat e}_{1}=\{W_{E_1},W_{E_2},W_{E_5}\}$. 
\begin{df}[Unit contraction]
    \Lh
    The \emph{unit-contraction} of $H$ is a hypergraph $\Bar H$ such that $V(\Bar H)=\mf{U}(H)$, and $E(\Bar H)=\{\hat{e}:e\in E(H)\}$.
\end{df}
 Since there exists a bijection $\pi: E(H)\to E(\Bar{H})$ defined by $e\mapsto \hat{e}$ for all $e\in E(H)$, two hypergraphs $H$ and $\Bar{H}$ have the same number of hyperedges.
 
 In the \Cref{sec-equi-rel}, we have described that given any equivalence relation $\mathfrak{R}$ on the vertex set $V(H)$, any complex-valued function $y: \mathfrak{C}_{\mathfrak{R}}\to\mathbb{C}$ corresponds to the $\mathfrak{R}$-blow-up $y_{_\mathfrak{R}}:V(H)\to\mathbb{C}$ such that $y_{_\mathfrak{R}}(v)=y(W)$, where $W$ is the ${\mathfrak{R}}$-equivalence class containing $v$. For $\mathfrak{R}=\mathfrak{R}_s$, the general signless Laplacian, general adjacency, and general Laplacian operators are $\mathfrak{R}_s$-compatible. Consequently, according to the \Cref{gen-contract-eig}, if $y:\mathfrak{U}(H)\to\mathbb{C}$ is an eigenvector corresponding to an eigenvalue $\lambda$ of the $\mathfrak{R}_s$-contraction of any of these operators, then  $\lambda$ is also an eigenvalue of the operator itself with eigenvector $ y_{_{\mathfrak{R}_s}}:V(H)\to\mathbb{C}$. We are going to use this fact to prove the next few results.
\subsection*{Units and general signless Laplacian operator.}
Now, we explore the spectra of $\q$ using $\mf{R}_s$. Our results describe the information of units encoded in the spectra of $\q$, and conversely, each unit has its contribution in the spectra of $\q$. We recall here that $(Q_{(H,\delta_{V(H)},\delta_{E(H)})}x)(v)=\sum\limits_{e\in E_v(H)}\frac{\delta_{E(H)}(e)}{\delta_{V(H)}(v)}\frac{1}{|e|^2}\sum\limits_{u\in e}x(u)=\sum\limits_{u\in V(H)}\sum\limits_{e\in E_u(H)\cap E_v(H)}\frac{\delta_{E(H)}(e)}{\delta_{V(H)}(v)}\frac{1}{|e|^2}x(u)$ for any vector $x:V(H)\to\mathbb{C}$.
Suppose that $T=\q$, we define $\hat{Q}_{H}:=T_{\mf{R}_s}$. In the proof of the next result, we show that if any $W_E\in\mf{U}(H)$ corresponds to a constant $c_{_E}$ such that  $\delta_{V(H)}(v)=c_{_E}$, for all $v\in W_E$, then $\q$ is $\mathfrak{R}_s$-compatible, and thus, \Cref{R-contract_exp} implies that, for all $y:\mf{U}(H)\to \mathbb{R}$,
$$ \hat{Q}_Hy=\sum\limits_{i,j=1}^mc^i_jy(W_{E_i})\epsilon_{W_{E_j}}.$$
Here, for any $v\in W_{E_j}$, the constant is 

$c^i_j=(\q\chi_{W_{E_i}})(v)=\sum\limits_{u\in V(H)}\sum\limits_{e\in E_u(H)\cap E_v(H)}\frac{\delta_{E(H)}(e)}{\delta_{V(H)}(v)}\frac{1}{|e|^2}\chi_{_{W_{_{E_i}}}}(u)=\sum\limits_{e\in E_i\cap E_j}\frac{\delta_{E(H)}(e)}{c_{_{E_j}}}\frac{1}{|e|^2}|W_{E_i}|$. Note that $\hat{Q}_H:\mathbb{C}^{V(\Bar{H})}\to \mathbb{C}^{V(\Bar{H})}$ is an operator associated with the unit contraction hypergraph $\Bar{H}$. 
\begin{thm}[Complete spectrum of $\q$ in terms of units]\label{cor_unit_eig}
\Lh
 If $\delta_{V(H)}(v)=c_{_E}$ for all $v\in W_E$, for any $W_E\in\mf{U}(H)$, then the complete spectrum of $\q$ is given below. 
\begin{enumerate}[leftmargin=*]
    \item  For each unit $W_E\in \mf{U}(H)$, $0$ is an eigenvalue of $\q$ with multiplicity $ |W_E|-1$, and $S_{W_E}$ is a subspace of eigenspace of the eigenvalue $0$.
    \item For each eigenvalue $\lambda$ of $\hat{Q}_H$ with eigenvector $y:\mf{U}(H)\to\mathbb{C}$, $\lambda$ is also an eigenvalue of $\q$ with the $\mf{R}_s$-blow up $y_{_{\mf{R}_s}}$ as an eigenvector.
\end{enumerate}
\end{thm}
\begin{proof}The matrix representation of $\q=[q_{uv}]_{u,v\in V(H)}$ is such that for all $u,v\in V(H)$, the entry $q_{uv}=\sum\limits_{e\in E_u(H)\cap E_v(H)}\frac{\delta_{E(H)}(e)}{\delta_{V(H)}(u)}\frac{1}{|e|^2}$. For $u,v\in W_E$, the star $E_u(H)=E=E_v(H)$.
Since  $\delta_{V(H)}(v)=c_{_E}$ for all $v\in W_E$, and , for any $u,v\in W_E$, $q_{uw}=\sum\limits_{e\in E\cap E_w(H)}\frac{\delta_{E(H)}(e)}{c_E}\frac{1}{|e|^2}=q_{vw}$, and $q_{wu}=\sum\limits_{e\in E_w(H)\cap E}\frac{\delta_{E(H)}(e)}{\delta_{V(H)}(w)}\frac{1}{|e|^2}=q_{wv}$, $q_{uu}=q_{vv}=q_{uv}=q_{vu}=\sum\limits_{e\in E}\frac{\delta_{E(H)}(e)}{c_E}\frac{1}{|e|^2}$. Therefore, $\mf{R}_s$ is finer than $\mf{R}_{\q}$. Thus, by the \Cref{equivalence-exm-prop} the operator $\q$ is $\mf{R}_s$-compatible.
    
  Since for any unit $W_E$ in $H$, we have $(\q x_{uv})(v)=\sum\limits_{e\in E_v(H)}\frac{\delta_{E(H)}(e)}{\delta_{V(H)(v)}}\frac{1}{|e|^2}\sum\limits_{w^\prime\in e}x_{uv}(w^\prime)=0$, for all $u,v\in W_E$. Thus, by \Cref{equivalence-eig}, part (1) of this result follows.

    The second part of the theorem follows from \Cref{gen-contract-eig}. The total number of eigenvalues provided by part (1) is $\sum\limits_{W_E\in \mf{U}(H)}(|W_E|-1)=|V(H)|-|\mf{U}(H)|$. Part (2) gives $|\mf{U}(H)|$ numbers of eigenvalues. Since all the eigenvectors provided in parts (1) and (2) are linearly independent, the list gives the complete spectrum of $\q$.
\end{proof}
The first part of \Cref{cor_unit_eig} is established using \Cref{equivalence-exm-prop} and \Cref{equivalence-eig}. However, an alternative proof can be derived by considering the matrix representation of $\q$. For each unit, $W_E$ with $\delta_{V(H)}(v)=c_{_E}$ for all $v\in W_E$, the matrix has $|W_E|$ numbers of identical rows, which implies that $0$ is an eigenvalue with a multiplicity of $|W_E| - 1$. But the use of \Cref{equivalence-eig} provides an idea of the eigenspace of the $0$ eigenvalue. This theorem ensures that the number of units with cardinality at least two in $H$ can be, at most, the multiplicity of the eigenvalue $0 $ of $Q_H$.
\begin{figure}[ht]
    \centering
    \begin{tikzpicture}[scale=0.8]
		\node [style=none] (0) at (-1, 0.25) {};
		\node [style=none] (1) at (-0.25, 2.25) {};
		\node [style=none] (2) at (1.25, 0) {};
		\node [style=none] (3) at (2.75, 3.25) {};
		\node [style=none] (4) at (-1, 0.25) {};
		\node [style=none] (5) at (0.5, 1.75) {};
		\node [style=none] (6) at (1.25, -0.75) {};
		\node [style=none] (7) at (3.75, 0.25) {};
		\node [style=none] (8) at (1.25, 1) {};
		\node [style=none] (9) at (0.25, -0.75) {};
		\node [style=none] (10) at (-1, 1.25) {};
		\node [style=none] (11) at (-2.75, -1.5) {};
		\node [style=new style 0,scale=0.5] (12) at (0, 0.75) {1};
		\node [style=new style 0,scale=0.5] (13) at (1.25, 2.5) {2};
		\node [style=new style 0,scale=0.5] (14) at (2.25, 2.25) {3};
		\node [style=new style 0,scale=0.5] (15) at (2.75, 0.5) {4};
		\node [style=new style 0,scale=0.5] (16) at (2.25, 0) {5};
		\node [style=new style 0,scale=0.5] (17) at (-1.75, -0.5) {6};
		\node [style=none] (18) at (-2.5, -1) {};
		\node [style=none] (19) at (-1.25, -0.25) {};
		\node [style=none] (20) at (2, -0.5) {};
		\node [style=none] (21) at (3, 1) {};
		\node [style=none] (22) at (-0.5, 0.25) {};
		\node [style=none] (23) at (0.5, 1.25) {};
		\node [style=none] (24) at (2.5, 2) {};
		\node [style=none] (25) at (1, 2.75) {};
		\node [style=none] (26) at (3.25, 3.5) {$e_1$};
		\node [style=none] (27) at (4.25, 0.25) {$e_2$};
		\node [style=none] (28) at (-2.8, 0) {$e_3$};
		\node [style=none] (29) at (0.5, -1.25) {H};
		\draw (2.center)
			 to [bend right=330, looseness=1.25] (0.center)
			 to [bend left] (1.center)
			 to [bend left, looseness=0.50] (3.center)
			 to [bend left, looseness=0.75] cycle;
		\draw (6.center)
			 to [bend right=330, looseness=1.25] (4.center)
			 to [bend left] (5.center)
			 to [bend left, looseness=0.75] (7.center)
			 to [bend left=15, looseness=0.75] cycle;
		\draw (10.center)
			 to [bend right=330, looseness=1.25] (8.center)
			 to [bend left] (9.center)
			 to [bend left, looseness=0.50] (11.center)
			 to [bend left, looseness=0.75] cycle;
		\draw [style=new edge style 1] (19.center)
			 to [bend right=60, looseness=1.25] (18.center)
			 to [bend right=60, looseness=1.25] cycle;
		\draw [style=new edge style 1] (21.center)
			 to [bend right=60, looseness=1.25] (20.center)
			 to [bend right=60, looseness=1.25] cycle;
		\draw [style=new edge style 1] (23.center)
			 to [bend right=60, looseness=1.25] (22.center)
			 to [bend right=60, looseness=1.25] cycle;
		\draw [style=new edge style 1] (25.center)
			 to [bend left=60, looseness=1.75] (24.center)
			 to [bend left=75, looseness=1.50] cycle;
    \node [style=none,scale=0.5] (30) at (0.25, 0) {$W_{E_2}$};
		\node [style=none,scale=0.5] (31) at (-1.75, -1.25) {$W_{E_1}$};
		\node [style=none,scale=0.5] (32) at (2.25, 0.5) {$W_{E_3}$};
		\node [style=none,scale=0.5] (33) at (1.5, 2) {$W_{E_4}$};
\end{tikzpicture}
    \caption{Hypergraph $H$ with $V(H)=\{1,2,3,4,5,6\}$, and $E(H)=\{e_1,e_2,e_3\}$. The units in $H$ are $W_{E_1}=\{6\}$, $W_{E_2}=\{1\}$, $W_{E_3}=\{4,5\}$, and $W_{E_4}=\{2,3\}$. The hyperedges are drawn with solid lines, whereas the units are illustrated with the dashed line.  }
    \label{fig:enter-label}
\end{figure}
For example, consider the matrix $Q_{(R,H)}=[q_{uv}]_{u,v\in V(H)}$ with $q_{uv}=|E_u(H)\cap E_v(H)|$ for $u,v\in V(H)$.  Suppose further that the underlying hypergraph is the hypergraph $H$ illustrated in the \Cref{fig:enter-label}. Recall the \Cref{cases}(1), the matrix $Q_{(R,H)}=\left[\begin{smallmatrix}
    3 & 1 & 1 & 1 & 1 & 1\\ 1 & 1 & 1 & 0 & 0 & 0\\ 1 & 1 & 1 & 0 & 0 & 0\\ 1 & 0 & 0 & 1 & 1 & 0\\ 1 & 0 & 0 & 1 & 1 & 0\\ 1 & 0 & 0 & 0 & 0 & 1 
\end{smallmatrix}\right]$ is a special case of $\q$ with $\delta_{V(H)}$ is constant in each unit. Two units $W_{E_3}$, and $W_{E_4}$ corresponds to two $0$ eigenvalues of $Q_{(R,H)}$ with eigenvectors $x_{23}$, and $x_{45}$. Here the matrix ${\hat Q}_{(R,H)}=\left[\begin{smallmatrix}
    3 & 2 & 2 & 1\\ 1 & 2 & 0 & 0\\ 1 & 0 & 2 & 0\\ 1 & 0 & 0 & 1
\end{smallmatrix}\right] $, which have eigenvalues $0,2,3+\sqrt{3},3-\sqrt{3}$. Thus, $Q_{(R,H)} $ also have these eigenvalues.

The following Theorem shows that the converse of the part (1) of the \Cref{cor_unit_eig} is true.
 \begin{thm}\label{converse-q-unit}
    Let $H $ be a hypergraph and $U(\subset V(H))$ with $|U|>1$. If $U$ is a maximal subset such that $S_U$ is a subspace of the eigenspace of the eigenvalue $0$ of $Q_H$, then $U$ is a unit.
    \end{thm}
    \begin{proof}
          If $u,v\in U$ then $x_{uv}=\chi_{\{v\}}-\chi_{\{u\}}\in S_U$. If a hyperedge $e$ is such that $u,v\in e $, then $\sum\limits_{w\in e}x_{uv}(w)=0$.
          Thus $0=(Q_Hx_{uv})(u)=\sum\limits_{e\in E_{u}(H)\setminus E_{v}(H)}\frac{\delta_{E(H)}(e)}{\delta_{V(H)}(u)}\frac{1}{|e|^2}$. Consequently, the set $ E_{u}(H)\setminus E_{v}(H)=\emptyset$. Therefore, $E_u(H)\subseteq E_v(H)$ and similarly, we can show $E_v(H)\subseteq E_u(H)$. Therefore, $E_u(H)=E_v(H)$ for all $u,v\in U$. Since $U$ is maximal subset with $E_u(H)=E_v(H)$ for all $u,v\in U$, thus $U$ is a unit.
    \end{proof}
   For instance, the subsets $\{2,3\}$, and $\{4,5\}$ are such that $U_i$ is maximal subset of vertices with $S_{U_i}$ is a subspace of the eigenspace of the eigenvalue $0$ for $i=1,2$. As suggested by the \Cref{converse-q-unit}, both $\{2,3\}$, and $\{4,5\}$ are units, and $\{2,3\}=W_{E_4}$, $\{4,5\}=W_{E_3}$, where $E_3=\{e_2\}$, $E_4=\{e_1\}$.
   \subsection*{Units and general Laplacian operator.} Recall that
   
   $ (L_{(H,\delta_{V(H)},\delta_{E(H)})}x)(v)=\sum\limits_{e\in E_v(H)}\frac{\delta_{E(H)}(e)}{\delta_{V(H)}(v)}\frac{1}{|e|^2}\sum\limits_{u\in e}(x(v)-x(u))$ for all vector $x:V(H)\to\mathbb{C}$.
We define $\hat{L}_H:=T_{\mf{R}_s}$, where $T=\lap$. If $\delta_{V(H)}(v)=c_{_E}$ for all $v\in W_E$, for any $W_E\in\mf{U}(H)$, then   we can show that $\lap$ is $\mathfrak{R}_s$-compatible. Consequently, \Cref{R-contract_exp} implies that, 
for all $y:\mf{U}(H)\to\mathbb{C}$,
$$ \hat{L}_Hy=\sum\limits_{i,j=1}^mc^i_jy(W_{E_i})\epsilon_{W_{E_j}},$$
where $c^i_j=(\lap\chi_{W_{E_i}})(v)=\sum\limits_{e\in E_j}\frac{\delta_{E(H)}(e)}{c_{_{E_j}}}\frac{1}{|e|^2}\left(\delta_{ij}|e|-|e\cap W_{E_i}|\right)$, $\delta_{ij}=1$ if $i=j$, otherwise $\delta_{ij}=0$. Here $\hat{L}_H:\mathbb{C}^{V(\Bar{H})}\to \mathbb{C}^{V(\Bar{H})}$ is an operator associated with the unit contraction hypergraph $\Bar{H}$. 
If we can prove $ \lap$ is  $\mf{R}_s$-compatible, then by \Cref{const-lem}, for any $W_E\in \mf{U}(H)$ with $|W_E|>1$ and $\delta_{V(H)}(v)=c_{_E}$ for all $v\in W_E$, it holds that the eigenvalue $\lambda_{(W_E,\lap)}=(\lap x_{uv})(v)=\sum\limits_{e\in E}\frac{\delta_{E(H)}(e)}{c_{_E}|e|}$ for all $u,v(\ne u)\in W_E$. Thus, \Cref{Tw} leads us to the following result.
\begin{thm}[Complete spectrum of $\lap$ in terms of units]\label{com-unil-l}\Lh
 If $\delta_{V(H)}(v)=c_{_E}$ for all $v\in W_E$, for any $W_E\in\mf{U}(H)$, then the complete spectrum of $\lap$ is the following. 
    \begin{enumerate}[leftmargin=*]
        \item  For each unit $W_E\in \mf{U}(H)$, $\sum\limits_{e\in E}\frac{\delta_{E(H)}(e)}{c_{_E}|e|}$ is an eigenvalue of $\lap$ with multiplicity $ |W_E|-1$, and $S_{W_E}$ is a subspace of eigenspace of the eigenvalue $\sum\limits_{e\in E}\frac{\delta_{E(H)}(e)}{c_{_E}|e|}$.
    \item For each eigenvalue $\lambda$ of $\hat{L}_H$ with eigenvector $y:\mf{U}(H)\to\mathbb{C}$, $\lambda$ is also an eigenvalue of $\lap$ with the $\mf{R}_s$-blow up $y_{_{\mf{R}_s}}$ as an eigenvector.
    \end{enumerate}
\end{thm}
\begin{proof} The operator $\lap$ can be proved $\mathfrak{R}_s$ compatible using the \Cref{equivalence-exm-prop} in the similar method that we have used in the proof of the \Cref{cor_unit_eig}. 
    Since $ \lambda_{(W_E,\lap)}=\sum\limits_{e\in E}\frac{\delta_{E(H)}(e)}{c_{_E}|e|}$ for all $W_E\in\mf{U}(H)$, by \Cref{Tw}, part (1) of this result follows.

    The second part of the theorem follows from \Cref{gen-contract-eig}. The total number of eigenvalues provided by part (1) is $\sum\limits_{W_E\in \mf{U}(H)}(|W_E|-1)=|V(H)|-|\mf{U}(H)|$. Part (2) gives $|\mf{U}(H)|$ numbers of eigenvalues. Since all the eigenvectors provided in parts (1) and (2) are linearly independent, the list gives the complete spectrum of $\lap$.
\end{proof}
As described in the \Cref{cases}(2), the matrix $L_{(B,H)}$ is a special case of $\lap$. For the hypergraph $H$, illustrated in the \Cref{fig:enter-label}, the matrix $L_{(B,H)}=\frac{1}{2}\left[\begin{smallmatrix}
\phantom{-} 6 & -1 & -1 & -1 & -1 & -2\\ -1 &\phantom{-} 2 & -1 & \phantom{-}0 & \phantom{-}0 & \phantom{-}0\\ -1 & -1 & \phantom{-}2 & \phantom{-}0 & \phantom{-}0 & \phantom{-}0\\ -1 & \phantom{-}0 & \phantom{-}0 & \phantom{-}2 & -1 & \phantom{-}0\\ -1 & \phantom{-}0 & \phantom{-}0 & -1 & \phantom{-}2 & \phantom{-}0\\ -2 & \phantom{-}0 & \phantom{-}0 & \phantom{-}0 & \phantom{-}0 & \phantom{-}2
\end{smallmatrix}\right]$. For this matrix $\delta_{V(H)}(v)=1$ for all $v\in V(H)$, and $\delta_{E(H)}(e)=\frac{|e|^2}{|e|-1}$ for all $e\in E(H)$. Thus, as suggested by the \Cref{com-unil-l}(1), the units $W_{E_3}$, and $W_{E_4}$ in $H$, respectively, corresponds to eigenvalues $\sum\limits_{e\in E_3}\frac{\delta_{E(H)}(e)}{c_{_{E_3}}|e|}=\frac{3}{2} $, and $\sum\limits_{e\in E_4}\frac{\delta_{E(H)}(e)}{c_{_{E_3}}|e|}=\frac{3}{2} $   of $L_{(B,H)}$ with eigenvectors $x_{23}$, and $x_{45}$. The eigenvalues of $\hat{L}_{(B,H)}=\frac{1}{2}\left[\begin{smallmatrix}
   \phantom{-} 6&-2&-2&-2\\-1&  \phantom{-}1&  \phantom{-}0&  \phantom{-}0\\-1&  \phantom{-}0&  \phantom{-}1&  \phantom{-}0\\-2&  \phantom{-}0&  \phantom{-}0&  \phantom{-}1
\end{smallmatrix}\right]$ are $0,\frac{1}{2},\frac{9-\sqrt{33}}{4},\frac{9+\sqrt{33}}{4}$, and by the \Cref{com-unil-l}(2), these are also eigenvalues of $L_{(B,H)}$.

\subsection*{Units and general adjacency operator.}Recall that for any $x:V(H)\to\mathbb{C}$, the general adjacency operator is  defined as $	(A_{(H,\delta_{V(H)},\delta_{E(H)})}x)(v)=\sum\limits_{e\in E_v(H)}\frac{\delta_{E(H)}(e)}{\delta_{V(H)}(v)}\frac{1}{|e|^2}\sum\limits_{u\in e;u\neq v}x(u).$
We define $\hat{A}_H:=T_{\mf{R}_s}$, where $T=\adj$.
If $\delta_{V(H)}(v)=c_{_E}$ for all $v\in W_E$, for any $W_E\in\mf{U}(H)$, then   we can show that $\adj$ is $\mathfrak{R}_s$-compatible. 
Thus,  if $\delta_{V(H)}(v)=c_{_E}$ for all $v\in W_E$, for any $W_E\in\mf{U}(H)$, then for all $y:\mf{U}(H)\to\mathbb{C}$,
$$ \hat{A}_Hy=\sum\limits_{i,j=1}^mc^i_jy(W_{E_i})\epsilon_{W_{E_j}},$$
where for some $v\in W_{E_j}, $ $c^i_j=(A_{(H,\delta_{V(H)},\delta_{E(H)})}\chi_{W_{E_i}})(v)=\sum\limits_{e\in E_j}\frac{\delta_{E(H)}(e)}{c_{_{E_j}}}\frac{1}{|e|^2}|W_{E_i}\cap e|$, if $i\ne j$,

otherwise $c^i_i=\sum\limits_{e\in E_i}\frac{\delta_{E(H)}(e)}{c_{_{E_i}}}\frac{1}{|e|^2}(|W_{E_i}\cap e|-1)$. Here $\hat{A}_H:\mathbb{C}^{V(\Bar{H})}\to \mathbb{C}^{V(\Bar{H})}$ is an operator associated with the unit contraction hypergraph $\Bar{H}$. 
Since $\adj$ is $\mf{R}_s$-compatible, by \Cref{const-lem}, if $W_E$ is a unit with $|W_E|>1$, and $\delta_{V(H)}(v)=c_{_E}$ for all $v\in W_E$, then $\lambda_{(W_E,\adj)}=(\adj x_{uv})(v)=-\sum\limits_{e\in E}\frac{\delta_{E(H)}(e)}{c_{_E}|e|^2}$. Thus, we have the following result. 
\begin{thm}[Complete spectrum of $\adj$ in terms of units]\label{com-unil-a}\Lh
 If $\delta_{V(H)}(v)=c_{_E}$ for all $v\in W_E$, for any $W_E\in\mf{U}(H)$, then the complete spectrum of $\adj$ is the following.
    \begin{enumerate}[leftmargin=*]
        \item  For each unit $W_E\in \mf{U}(H)$, $-\sum\limits_{e\in E}\frac{\delta_{E(H)}(e)}{c_{_E}|e|^2}$ is an eigenvalue of $\adj$ with multiplicity $ |W_E|-1$, and $S_{W_E}$ is a subspace of eigenspace of the eigenvalue $-\sum\limits_{e\in E}\frac{\delta_{E(H)}(e)}{c_{_E}|e|^2}$.
    \item For each eigenvalue $\lambda$ of $\hat{A}_H$ with eigenvector $y:\mf{U}(H)\to\mathbb{C}$, $\lambda$ is also an eigenvalue of $\adj$ with the $\mf{R}_s$-blow up $y_{_{\mf{R}_s}}$ as an eigenvector.
    \end{enumerate}
\end{thm}
\begin{proof}The operator $\lap$ can be proved $\mathfrak{R}_s$ compatible using the \Cref{equivalence-exm-prop} in the similar method that we have used in the proof of the \Cref{cor_unit_eig}. 
    Since $ \lambda_{(W_E,\lap)}=\sum\limits_{e\in E}\frac{\delta_{E(H)}(e)}{c_{_E}|e|}$ for all $W_E\in\mf{U}(H)$, by \Cref{equivalence-eig}, part (1) of this result follows.

    The second part of the theorem follows from \Cref{gen-contract-eig}. The total number of eigenvalues provided by part (1) is $\sum\limits_{W_E\in \mf{U}(H)}(|W_E|-1)=|V(H)|-|\mf{U}(H)|$. Part (2) gives $|\mf{U}(H)|$ numbers of eigenvalues. Since all the eigenvectors provided in parts (1) and (2) are linearly independent, the list gives the complete spectrum of $\adj$.
\end{proof}
In the \Cref{cases}, we show that $A_{(N,H)}$ is a variation of the $\adj$ in which $\delta_{V(H)}(v)=|E_v(H)|$ for all $v\in V(H)$, and $\delta_{E(H)}(e)=\frac{|e|^2}{|e|-1}$ for all $e\in E(H)$. For the hypergraph $H$ illustrated in the \Cref{fig:enter-label}, the matrix $A_{(N,H)}=\frac{1}{6}\left[\begin{smallmatrix}
    0 & 1 & 1 & 1 & 1 & 2\\ 3 & 0 & 3 & 0 & 0 & 0\\ 3 & 3 & 0 & 0 & 0 & 0\\ 3 & 0 & 0 & 0 & 3 & 0\\ 3 & 0 & 0 & 3 & 0 & 0\\ 6 & 0 & 0 & 0 & 0 & 0 
\end{smallmatrix}\right] $. 
As suggested by the \Cref{com-unil-a}, the units $W_{E_3}$, and $W_{E_4}$, respectively, corresponds to the eigenvalues $-\sum\limits_{e\in E_3}\frac{\delta_{E(H)}(e)}{c_{_{E_3}}|e|^2}=-\frac{1}{2}$, and $-\sum\limits_{e\in E_4}\frac{\delta_{E(H)}(e)}{c_{_{E_4}}|e|^2}=-\frac{1}{2}$ of $A_{(N,H)}$. The matrix $\hat{A}_{(N,H)}=\frac{1}{6}\left[\begin{smallmatrix}
    0 & 2 & 2 & 2\\ 3 & 3 & 0 & 0\\ 3 & 0 & 3 & 0\\ 6 & 0 & 0 & 0
\end{smallmatrix}\right]$, and by \Cref{com-unil-a}(2), each eigenvalue of this matrix is an eigenvalue of $A_{(N,H)} $.
   \subsection{Twin units}\label{sec-twin}
   In part (2) of each of the results \Cref{cor_unit_eig}, \Cref{com-unil-l}, and \Cref{com-unil-a}, we have provided quotient operators of the general operators associated with the hypergraph and showed that the spectra of the quotient operators are subsets of the spectra of these general operators. The quotient matrices may not be always small in size.
  In the worst case, they can be equal to the matrix representations of the general operators. Using units in this subsection, we introduce a new building block that can further help us to compute some eigenvalues of the quotient operators.  
 
\begin{df}[Twin units]
    \Lh
    For two distinct $W_E,W_F\in\mf{U}(H)$, if there exists a bijection $\mf{f}_{EF}:E\to F$ such that 
    $$\mf{f}_{EF}(e)=(e\setminus W_E)\cup W_F\text{~for all~}e\in E, $$ then we refer $\mf{f}_{EF}$ to as the \emph{canonical bijection} between $E$, and $F$. A pair of units $W_E,$ and $W_F$ are called \emph{twin units} if a canonical bijection exists between their generators $E,$ and $F$.
\end{df}
For example, in the hypergraph $H$, given in \Cref{fig:unit}, $W_{E_6}$, and $W_{E_7}$ are twin units. The canonical bijection $\mf{f}_{E_6E_7}:E_6\to E_7$ is defined by $e_4\mapsto e_5$, and $e_6\to e_7$. The pair $W_{E_8}$, and $W_{E_1}$ are twin units. The canonical bijection $\mf{f}_{E_8E_1}:E_8\to E_1$ is defined by $e_6\mapsto e_4$, and $e_7\to e_5$. Similarly, any two units in $\{W_{E_3},W_{E_4},W_{E_5}\}$ are twin units.

	\begin{prop}\label{twin_nbd}
		Twin units can not be unit-neighbours to each other.
	\end{prop}
	\begin{proof}
	Suppose that $W_{E_1},W_{E_2}(\ne W_{E_1})\in \mathfrak{U}(H)$ are twin units. If possible, let us assume that $W_{E_1}, W_{E_2}$ are unit neighbours. Then, there exists $e\in E(H)$ such that $W_{E_1}\cup W_{E_2}\subseteq e$. Now, there exists $f\in E_2$ such that $e\setminus W_{E_1}=f\setminus W_{E_2}$. Since $W_{E_1}\cup W_{E_2}\subseteq e$, we have $W_{E_2}\subseteq e\setminus W_{E_1}$. This contradicts $e\setminus W_{E_1}=f\setminus W_{E_2}$. Therefore, our assumption is wrong. Hence the proof follows.
	\end{proof}
The notion of twin units
yields an equivalence relation $\mathfrak{R}_{tw}$ on $\mathfrak{U}(H)$, defined by
	$$\mathfrak{R}_{tw}=\{(W_{E_1},W_{E_2})\in \mathfrak{U}(H)^2:W_{E_1},W_{E_2} \text{ are twin units}\}.$$
  We denote the $\mathfrak{R}_{tw}$-equivalence class of a unit $W_{E_i}$ 
as $\mathfrak{C}(W_{E_i})$. We denote the collection of all $\mathfrak{R}_{tw}$-equivalence classes in $\mathfrak{U}(H)$ by $\mathfrak{C}(\mathfrak{U}(H))$.  
\subsection*{Twin units and spectra of hypergraphs}
   Like units, twin units are also symmetric structures. Thus, we have the following Theorem.
   \begin{thm}\label{twin-adj}
       Let $H$ be a hypergraph such that $W_{E_i}$ and $W_{E_j}$ are twin units in $H$. If $\delta_{V(H)}(v)=c$, a constant for all $v\in W_{E_i}\cup W_{E_j} $, and  $\frac{\delta_{E(H)}(e)}{|e|^2}=w$, a constant for all $e\in E_i\cup E_j$ then $\frac{w}{c}\sum\limits_{e\in E_i}|e\setminus W_{E_i}|$ is an eigenvalue of $L_{(H,\delta_{V(H)},\delta_{E(H)})}$ with the eigenvector $y=|W_{{E_{i}}}|\chi_{_{W_{E_j}}}-|W_{E_{j}}|\chi_{_{W_{E_{i}}}}$.
   \end{thm}
   \begin{proof}
       For any $v\in V(H)$, suppose that $E^k_v(H)=E_v(H)\cap E_k$ for $k=i,j$. Since $W_i$ and $W_{E_j}$ are twin units, $E_i\cap E_j=\emptyset$. Thus,
	    \begin{align*}
	       (L_{(H,\delta_{V(H)},\delta_{E(H)})}y)(v)
	        &=\sum\limits_{e\in E^i_v(H)}\frac{\delta_{E(H)}(e)}{\delta_{V(H)}(v)}\frac{1}{|e|^2}\sum\limits_{u\in e}(y(v)-y(u))\\
	        &+\sum\limits_{e\in E^{j}_v(H)}\frac{\delta_{E(H)}(e)}{\delta_{V(H)}(v)}\frac{1}{|e|^2}\sum\limits_{u\in e}(y(v)-y(u)).
	    \end{align*} 
        For any $v\in V(H)\setminus (W_{E_i}\cup W_{E_{j}})$, considering the restriction of the canoncal map $$\mathfrak{f}=\mathfrak{f}_{E_iE_{j}}|_{E_v^i(H)}:E_v^i(H)\to E_v^j(H),$$ we have 
	    \begin{align*}
	        (L_{(H,\delta_{V(H)},\delta_{E(H)})}y_j)(v)&=\frac{w_a}{\delta_{V(H)}(v)}\left(\sum\limits_{e\in E^i_v(H)}|W_{E_i}||W_{E_{j}}|-\sum\limits_{\mathfrak{f}(e)\in E^j_v(H)}|W_{E_i}||W_{E_{j}}|\right)=0.
	    \end{align*}
	    For any $v\in W_{E_j}$, we have, $$E_v(H)=E_v^j(H)=E_j,E^{i}_v=\emptyset. $$ Since $e\setminus W_{E_j}=\mathfrak{f}(e)\setminus W_{E_{i}}$,
        we have 

\begin{align*}
    (L_{(H,\delta_{V(H)},\delta_{E(H)})}y_j)(v) &=\frac{w}{c}\sum\limits_{e\in E_j}|e\setminus W_{E_j}|y_j(v)\\&=\frac{w}{c}\sum\limits_{e\in E_{i}}|e\setminus W_{E_{i}}|y_j(v).
\end{align*}
Similarly, for any $v\in W_{E_i}$, we have, $ (L_{(H,\delta_{V(H)},\delta_{E(H)})}y)(v) =\frac{w}{c}\sum\limits_{e\in E_{i}}|e\setminus W_{E_{i}}|y_j(v)$. This completes the proof.
   \end{proof}
  The matrix  $L_{(B,H)}$ is a special case of $\lap$. For the hypergraph $H$, illustrated in the \Cref{fig:enter-label}, the units $W_{E_3}$, and $W_{E_4}$ are a pair of twin units. Here $w=\frac{1}{2}$, $c=1$, and consequently by the \Cref{twin-adj}, the matrix $L_{(B,H)}$ has an eigenvalue $\frac{w}{c}\sum\limits_{e\in E_3}|e\setminus W_{E_3}|= \frac{1}{2}$.
   The previous theorem leads us to the following result. Given any  $\mathfrak{R}_{tw}$-equivalence class $a\in \mathfrak{C}(\mathfrak{U}(H))$, we denote the number of units in $a$ as $n_a$.
  	\begin{cor}\label{lap_twin_unit}
	    Let $H$ be a hypergraph such that $a=\{W_{E_1},W_{E_2},\ldots,W_{E_{n_a}}\}\in \mathfrak{C}(\mathfrak{U}(H))$, $n_a>1$. If $\delta_{V(H)}(v)=c_a$ for all $v\in \bigcup\limits_{i=1}^{n_a} W_{E_i} $, $\frac{\delta_{E(H)}(e)}{|e|^2}=w_a$ for all $e\in E_a= \bigcup\limits_{i=1}^{n_a} E_i$ then $\frac{w_a}{c_a}\sum\limits_{e\in E_{n_a}}|e\setminus W_{E_{n_a}}|$ is an eigenvalue of $L_{(H,\delta_{V(H)},\delta_{E(H)})}$ with multiplicity at least $n_a-1$.
	\end{cor}
	We provide similar results for $A_{(H,\delta_{V(H)},\delta_{E(H)})}$ and $Q_{(H,\delta_{V(H)},\delta_{E(H)})}$ in the next result.

\begin{thm}\label{twin-units-adj}
       Let $H$ be a hypergraph such that $W_{E_i}$ and $W_{E_j}$ are a pair of equipotent twin units in $H$ with $|W_{E_i}|=|W_{E_j}|=m$. If $\delta_{V(H)}(v)=c$, a constant for all $v\in W_{E_i}\cup W_{E_j} $, and  $\frac{\delta_{E(H)}(e)}{|e|^2}=w$, a constant for all $e\in E_i\cup E_j$ then $\frac{w}{c}s(m-1)$ and $\frac{w}{c}ms$ is an eigenvalue of $A_{(H,\delta_{V(H)},\delta_{E(H)})}$ and $Q_{(H,\delta_{V(H)},\delta_{E(H)})}$ respectively with the eigenvector $y=|W_{{E_{i}}}|\chi_{_{W_{E_j}}}-|W_{E_{j}}|\chi_{_{W_{E_{i}}}}$, where $s=|E_i|$.
   \end{thm}
 \begin{proof}
Let $E^k_v(H)=E_v(H)\cap E_k$ for $k=i,j$ for all $v\in V(H)$. Thus,
     \begin{align*}
	        (A_{(H,\delta_{V(H)},\delta_{E(H)})}y)(v)
	        &=\sum\limits_{e\in E^j_v(H)}\frac{\delta_{E(H)}(e)}{\delta_{V(H)}(v)}\frac{1}{|e|^2}\sum\limits_{u\in e;u\neq v}y(u)\\+&\sum\limits_{e\in E^{i}_v(H)}\frac{\delta_{E(H)}(e)}{\delta_{V(H)}(v)}\frac{1}{|e|^2}\sum\limits_{u\in e;u\neq v}y(u),
	    \end{align*}
      Thus, similarly, in the previous proof, we have $(A_{(H,\delta_{V(H)},\delta_{E(H)})}y_j)(v)=0$. Proceeding in the same way, we have $(Q_{(H,\delta_{V(H)},\delta_{E(H)})}y_j)(v)=0$.
     For any $v\in W_{E_j}$, we have, $E_v(H)=E_v^j(H)=E_j,E^{i}_v(H)=\emptyset $, and
       \begin{align*}
         (A_{(H,\delta_{V(H)},\delta_{E(H)})}y)(v) &=\frac{w}{c}|E_j|(|W_{E_j}|-1)y_j(v)\\&=\frac{w}{c}s(m-1)y(v).
     \end{align*}
	 Similarly, for any $ v\in W_{E_{i}}$, $E_v(H)=E_v^{i}(H)=E_{i},E^j_v=\emptyset $. Since $e\setminus W_{E_j}=\mathfrak{f}(e)\setminus W_{E_{i}}$ for all $e\in E_j$, we have $$A_{(H,\delta_{V(H)},\delta_{E(H)})}y(v)=\frac{w}{c}s(m-1)y(v).$$
  
  Therefore, $\frac{w}{c}s(m-1) $ is an eigenvalue of $A_{(H,\delta_{V(H)},\delta_{E(H)})}$.
  Since $E_v(H)=E_j$ for each $v\in W_{E_j}$ and $\sum\limits_{u\in e}y(u)=m$ for all $e\in E_j$, we have $$(Q_{(H,\delta_{V(H)},\delta_{E(H)})}y)(v)=m|E_j|\frac{w}{c}y(v)=ms\frac{w}{c}y(v).$$
    Similarly, we can prove the same for $v\in W_{E_{i}}$. Therefore $(Q_{(H,\delta_{V(H)},\delta_{E(H)})}y)=ms\frac{w}{c}y$.
	    
 \end{proof}
As described in the \Cref{cases}, $A_{(N,H)}$ is a special case of the $\adj$. In the hypergraph $H$ illustrated in \Cref{fig:enter-label}, for the pair of twin units $W_{E_3}$, and $W_{E_4}$ we have $w=\frac{1}{2}$, $c=1$, $m=2$, and $s=1$. Consequently, the pair of twin units corresponds to the eigenvalue $\frac{1}{2}$ of the eigenvalue $A_{(N,H)}$. The matrix $Q_{(R,H)}$ is a special case of $\q$. For the pair of twin units $W_{E_3}$, and $W_{E_4}$ in the hypergraph $H$ illustrated in the \Cref{fig:enter-label} , we have $w=1$, $c=1$, $ m=2$, $s=1$. Therefore, by the \Cref{twin-units-adj}, $\frac{w}{c}ms=2$ is an eigenvalue of  $Q_{(R,H)}$.
 The next result directly follows from the above Theorem. For any $a=\{W_{E_1},W_{E_2},\ldots,W_{E_{n_a}}\}\in \mathfrak{C}(\mathfrak{U}(H))$ it holds that $|E_1|=|E_2|=\ldots=|E_{n_a}|$. That is , there exists a constant $s_a$ such that $|E_i|=s_a$ for all $i=1,\ldots,n_a$. We refer to $s_a$ as the generator size of $a$.
		\begin{cor}\label{adj_twin_unit}
	    Let $H$ be a hypergraph such that $a=\{W_{E_1},W_{E_2},\ldots,W_{E_{n_a}}\}\in \mathfrak{C}(\mathfrak{U}(H))$, $n_a(\in\n)>1$, and  $|W_{E_i}|=m_a$ for all $W_{E_i}\in a$. If $\delta_{V(H)}(v)=c_a$ for all $v\in \bigcup\limits_{i=1}^{n_a} W_{E_i} $, $\frac{\delta_{E(H)}(e)}{|e|^2}=w_a$ for all $e\in E_a= \bigcup\limits_{i=1}^{n_a} E_i$ then $\frac{w_a}{c_a}s_a(m_a-1)$ and $\frac{w_a}{c_a}m_as_a$ are eigenvalues of $A_{(H,\delta_{V(H)},\delta_{E(H)})}$ and $Q_{(H,\delta_{V(H)},\delta_{E(H)})}$ respectively with multiplicity at least $n_a-1$, where $s_a$ is the generator size of $a$.
             \end{cor}

   The eigenvectors provided in \Cref{lap_twin_unit}, \Cref{adj_twin_unit} are linearly independent with the eigenvectors provided in part (1) of each of the \Cref{cor_unit_eig}, \Cref{com-unil-l}, and \Cref{com-unil-a}.
   For any hypergraph $H$ with $\mathfrak{C}(\mathfrak{U}(H))=\{a_1,a_2,\ldots,a_m\}$, $a_i=\{W_{E^i_1},W_{E^i_2},\ldots,W_{E^i_{n_i}}\}$.
   
Since \Cref{cor_unit_eig} provides $\sum\limits_{W_{E_i}\in \mathfrak{U}(H)}(|W_{E_i}|-1)$ eigenvalues, \Cref{lap_twin_unit} provides $\sum\limits_{a\in\mathfrak{C}(\mathfrak{U}(H))}(|a|-1)$ eigenvalues the remaining number of eigenvalue of $L_{(H,\delta_{V(H)},\delta_{E(H)})}$ are $|\mathfrak{C}(\mathfrak{U}(H))|$. If all the units of $H$ are of the same cardinality, then \Cref{lap_contract} provide us with the remaining eigenvalues. Similarly, our results provide the complete spectra of $Q_{(H,\delta_{V(H)},\delta_{E(H)})}$.

 For any hypergraph $H$, the collection of all the units $\mathfrak{U}(H)$ forms a partition of $V(H)$ and $\mathfrak{C}(\mathfrak{U}(H))$ forms an partition of $ \mathfrak{U}(H)$. These two partitions naturally induce an onto map. We define $\pi:V(H)\to\mathfrak{C} (\mathfrak{U}(H))$ as 
	$\pi(v)=\mathfrak{C}(W_{E_v(H)})$. Since for any $\mathfrak{R}_{tw}$-equivalence class $a\in \mathfrak{C} (\mathfrak{U}(H))$, there exists an unit $W_E\in a$, and the unit $W_E$ contain at least one vertex, say $v$. Since $\pi(v)=a$, the map $\pi$ is onto.
	\begin{prop}\label{quotient-prop}
		Let $H$ be a hypergraph. If $e=\bigcup\limits_{i=1}^n W_{E_i}\in E(H)$ then $e^\prime=\bigcup\limits_{i=1}^n W_{E_i^\prime}\in E(H)$ for all $W_{E_i^\prime}\in \mathfrak{C}(W_{E_i})$, $i=1,2,\ldots,n$.
	\end{prop}
	\begin{proof}
		Since $W_{E_i^\prime}\in \mathfrak{C}(W_{E_i})$, $W_{E_i(H)}$ and $W_{E_i^\prime(H)}$ are twin units and there exists a canonical bijection $\mathfrak{f}_{E_iE_i^\prime}$. Therefore, $e^\prime =(\mathfrak{f}_{E_nE_n^\prime}\circ\ldots\circ\mathfrak{f}_{E_2E_2^\prime}\circ\mathfrak{f}_{E_1E_1^\prime})(e)\in E(H)$.
	\end{proof}
The above result leads us to the following definition.
	\begin{df} Let $H$ be a hypergraph. For all $e\in E(H)$,
		$\hat{\pi}(e)=\{\pi(v):v\in e\}$. The contraction of $H$ is a hypergraph, denoted by $\hat H$, is defined as $V(\hat H)=\mathfrak{C}(\mathfrak{U}(H))$ and $E(\hat H)=\{\hat{\pi}(e):e\in E(H)\}$.
	\end{df}
  For instance, consider the hypergraph $H$ illustrated in the \Cref{fig:enter-label}. Since the hypergraph contains only a pair of twin units $W_{E_3}$, and $W_{E_4}$. The contraction $\hat H$ is such that $V(\hat H)=\{\mathfrak{C}(W_{E_1}),\mathfrak{C}(W_{E_2}), \mathfrak{C}(W_{E_3})\}$, and $E(\hat H)=\{\hat{\pi}(e_1)=\hat{\pi}(e_2),\hat{\pi}(e_3)\}$.

Given any hypergraph $H$, for any $ x\in\mathbb{C}^{V(\hat H)}$, we define $ \Bar x\in \mathbb{C}^{V( H)}$ as $\Bar x(v)={x}(\pi(v))$. For any $ y\in\mathbb{C}^{V(\hat H)}$, we define $ \grave y\in \mathbb{C}^{V( H)}$ as $\grave y(v)=\frac{1}{|W_{E_v(H)}|}y(\mathfrak{C}(W_{E_v(H)})) $ for all $v\in V(H)$. For any $ \alpha \in \mathbb{C}^{E({H})} $, we define $ \hat\alpha \in \mathbb{C}^{E(\hat{H})} $ as $\hat{\alpha}(\hat \pi(e))=\sum\limits_{e\in \hat{\pi}^{-1}(\hat{\pi}(e))}\alpha(e)$.

 Before the next result, we recall that functions $\sigma_H:E(H)\to(0,\infty)$ and $ \rho_H:E(H)\to(0,\infty)$ are defined as $\sigma_H(e)=\frac{\delta_{E(H)}(e)}{|e|^2}$ and $\rho_H(e)=\frac{\delta_{E(H)}(e
		)}{|e|}$  respectively, where $\delta_{V(H)}$, and $\delta_{E(H)}$ are the positive valued function on $V(H)$, and $E(H)$, respectively, used to define the general operators associated with a hypergraph.
 
	   \begin{lem}\label{contrac-lap-lem}
Let $H$ be a hypergraph. If $\delta_{V(H)}=c_V{\bar\delta_{V({\hat H})}}$, ${\hat\sigma_H}=c_{_E}\sigma_{\hat H}$, $ |W_{E_0}|=c$ for all $ W_{E_0}\in \mathfrak{U}(H)$, and the cardinality, $|a|=k$ for all $a\in\mathfrak{C}(\mathfrak{U}(H)) $ then $(L_{(H,\delta_{V(H)},\delta_{E(H)})}\grave y)(v)= k\frac{c_{_E}}{c_V}(L_{(\hat{H},\delta_{V(\hat H)},\delta_{E(\hat H)})}y)(\pi(v))$ and $(Q_{(H,\delta_{V(H)},\delta_{E(H)})}\grave y)(v)=k\frac{c_{_E}}{c_V}(Q_{(\hat{H},\delta_{V(\hat H)},\delta_{E(\hat H)})}y)(\pi(v))$ for any $y\in\mathbb{C}^{V(\hat H)}$, and $v\in V(H)$.
\end{lem}
\begin{proof}
    For all $y\in\mathbb{C}^{V(\hat H)}$, and $v\in V(H)$ we have 
    \begin{align*}
        (L_{(H,\delta_{V(H)},\delta_{E(H)})}\grave y)(v)&=\sum\limits_{e\in E_v(H)}\frac{\sigma_H(e)}{\delta_{V(H)}(v)}\sum\limits_{u\in e}(\grave y(v)-\grave y(u))\\
        &=\sum\limits_{\hat\pi(e)\in E_{\pi(v)(H)}}\frac{\hat\sigma_H(\hat\pi(e))}{c_V\bar\delta_{V(\hat H)}(v)}k\sum\limits_{\pi(u)\in \hat \pi(e)}( y(\pi(v))- y(\pi(u)))\\
         &=\sum\limits_{\hat\pi(e)\in E_{\pi(v)(H)}}\frac{c_{_E}\sigma_{\hat H}(\hat\pi(e))}{c_V\bar\delta_{V(\hat H)}(v)}k\sum\limits_{\pi(u)\in \hat \pi(e)}( y(\pi(v))- y(\pi(u)))\\
         &=k\frac{c_{_E}}{c_V}(L_{(\hat{H},\delta_{V(\hat H)},\delta_{E(\hat H)})}y)(\pi(v))
    \end{align*}
    Similarly, we have $(Q_{(H,\delta_{V(H)},\delta_{E(H)})}\grave y)(v)=k\frac{c_{_E}}{c_V}(Q_{(\hat{H},\delta_{V(\hat H)},\delta_{E(\hat H)})}y)(\pi(v))$.
\end{proof}
Therefore, we have the following result by \Cref{contrac-lap-lem}.
\begin{thm}\label{lap_contract}
Let $H$ be a hypergraph. If $\delta_{V(H)}=c_V{\bar\delta_{V({\hat H})}}$, ${\hat\sigma_H}=c_{_E}\sigma_{\hat H}$, $ |W_{E_0}|=c$ for all $ W_{E_0}\in \mathfrak{U}(H)$ , and the cardinality, $|a|=k$ for all $a\in\mathfrak{C}(\mathfrak{U}(H)) $, then for each eigenvalue $\lambda_{L_{\hat H}}$ of $L_{(\hat{H},\delta_{V(\hat H)},\delta_{E(\hat H)})}$ with eigenvector $y$ we have $k \frac{cc_{_E}}{c_V}\lambda_{L_{\Bar{H}}}$ is an eigenvalue of $L_{(H,\delta_{V(H)},\delta_{E(H)})}$ with eigenvector $\grave y$, and for each eigenvalue $\lambda_{Q_{\hat H}}$ of $Q_{(\hat{H},\delta_{V(\hat H)},\delta_{E(\hat H)})}$ with eigenvector $z$ we have $k \frac{cc_{_E}}{c_V}\lambda_{Q_{\Bar{H}}}$ is an eigenvalue of $Q_{(H,\delta_{V(H)},\delta_{E(H)})}$ with eigenvector $\grave z$. \end{thm}
\begin{proof}
    By \Cref{contrac-lap-lem}, we have  $$(L_{(H,\delta_{V(H)},\delta_{E(H)})}\grave y)(v)= k\frac{c_{_E}}{c_V}(L_{\Hat{H}}y)(\pi(v))$$ for any $y\in\mathbb{C}^{V(\hat H)}$, and $v\in V(H)$.
    Since $ \lambda_{L_{\hat H}}$ is an eigenvalue with eigenvector $y$ and $$\grave y(v)=\frac{1}{|W_{E_v(H)}|}y(\mathfrak{C}(W_{E_v(H)}))=\frac{1}{c}y(\mathfrak{C}(W_{E_v(H)})) $$
    
    for all $v\in V(H)$, we have 
   \begin{align*}
       (L_{(H,\delta_{V(H)},\delta_{E(H)})}\grave y)(v)&=k \frac{c_{_E}}{c_V}(L_{\Hat{H}}y)(\pi(v))\\&=\frac{c_{_E}}{c_V} \lambda_{L_{\hat H}}y(\pi(v))\\&=k\frac{cc_{_E}}{c_V} \lambda_{L_{\hat H}}\grave y(v)
   \end{align*}   
    for any $y\in\mathbb{C}^{V(\hat H)}$, and $v\in V(H)$.
    The proof for the part of $Q_{(H,\delta_{V(H)},\delta_{E(H)})}$ is similar to that of $L_{(H,\delta_{V(H)},\delta_{E(H)})}$ and hence is omitted.
\end{proof}
Let $ \mathfrak{C}(\mathfrak{U}(H))=\{a_1,\ldots,a_k\}$ for a hypergraph $H$. Suppose that $\delta_{V(H)}(v)=c_p$ for all $v\in \pi^{-1}(a_p)$, all canonical bijections are $\sigma_H$ preserving, and $|W_{E_i}|=m_p$ for all $W_{E_i}\in a_p$  for all $p=1,\ldots,k$. We define $C_H=\left(c_{pq}\right)_{a_p,a_q\in \mathfrak{C}(\mathfrak{U}(H))}$ as $c_{pq}=\frac{1}{c_q}| \pi^{-1}(a_p)|\sum\limits_{E_i\cap E_j}\sigma_H(e)$, where $a_p=\mathfrak{C}(W_{E_i}),a_q=\mathfrak{C}(W_{E_j})$ and $p\ne q$ and $c_{pp}=(m_p-1)\sum\limits_{v\in E_i}\sigma_H(e)$.
\begin{thm}\label{twin-eqp-eig}
Let $ \mathfrak{C}(\mathfrak{U}(H))=\{a_1,\ldots,a_k\}$ for a hypergraph $H$. If $\delta_{V(H)}(v)=c_p$ for all $v\in \pi^{-1}(a_p)$, all canonical bijections are $\sigma_H$ preserving, and $|W_{E_i}|=m_p$ for all $W_{E_i}\in a_p$  for all $p=1,\ldots,k$ then for each eigenvalue $ \lambda$ of $C_H$ with eigenvector $y$ we have $\lambda$ is an eigenvalue of $A_{(H,\delta_{V(H)},\delta_{E(H)})}$ with eigenvector $ \Tilde{y}$ defined by $\Tilde{y}(v)=\frac{y(a_p)}{c_p}$, for all $v\in \pi^{-1}(a_p)$, $p=1,\ldots,k$.
\end{thm}
\begin{proof}
For all $v\in \pi^{-1}(a_p)$, 
\begin{align*}
	(A_{(H,\delta_{V(H)},\delta_{E(H)})}\Tilde{y})(v)
	&=\frac{1}{\delta_{V(H)}(v)}\sum\limits_{u(\ne)v\in V(H)}\Tilde{y}(u)\sum\limits_{e\in E_{uv}(H)}\sigma_H(e)\\
	&=\frac{1}{\delta_{V(H)}(v)}\sum\limits_{a_q\in\mathfrak{C}(\mathfrak{U}(H))}y(a_q)c_{pq}=\frac{1}{\delta_{V(H)}(v)}((C_H)y)(a_p)\\
	&=\frac{\lambda}{\delta_{V(H)}(v)} y(a_p)=\lambda\Tilde{y}(v).
\end{align*}
Thus, the Theorem follows.
\end{proof}
 The matrices $Q_{(B,H)}$ is a special cases of $\q$ with $\delta_{V(H)}(v)=v$ for all $v\in V(H)$, and $\delta_{E(H)}=\frac{|e|^2}{|e|-1}$ for all $e\in E(H)$. Consider the hypergraph $H$ with $V(H)=\{1,2,3,4,5,6,7,8\}$, and $E(H)=\{e_1=\{1,2,3,4\},e_2=\{1,2,7,8\},e_3=\{3,4,5,6\},e_4=\{5,6,7,8\}\}$. The units in $H$ are $W_{E_1}=\{1,2\}$, $W_{E_2}=\{5,6\}$, $W_{E_3}=\{3,4\}$, and $W_{E_4}=\{7,8\}$. For this hypergraph $H$, we have $ Q_{(B,H)}=\frac{1}{3}\left[\begin{smallmatrix}
     2 & 2 & 1 & 1 & 0 & 0 & 1 & 1\\ 2 & 2 & 1 & 1 & 0 & 0 & 1 & 1\\ 1 & 1 & 2 & 2 & 1 & 1 & 0 & 0\\ 1 & 1 & 2 & 2 & 1 & 1 & 0 & 0\\ 0 & 0 & 1 & 1 & 2 & 2 & 1 & 1\\ 0 & 0 & 1 & 1 & 2 & 2 & 1 & 1\\ 1 & 1 & 0 & 0 & 1 & 1 & 2 & 2\\ 1 & 1 & 0 & 0 & 1 & 1 & 2 & 2 
 \end{smallmatrix}\right]$. By the \Cref{cor_unit_eig}, $Q_{(B,\hat H)}$ has the eigenvalue $0$ with multiplicity at least $4$ for the units. By the \Cref{twin-units-adj}, this matrix has an eigenvalue $\frac{4}{3}$ of multiplicity $2$ due to the twin units in $H$. Since here $ Q_{(B,\hat H)}=\left[\begin{smallmatrix}
     1&1\\1&1
 \end{smallmatrix}\right]$, and the eigenvalue of $Q_{(B,\hat H)}$ are $0$, and $2$. Here $k=2$, $c=2$, and $ c_E=\frac{1}{3}$. Thus, for this hypergraph, and this matrix $k\frac{cc_E}{c_V}=\frac{4}{3}$, the matrix $Q_{(B,H)}$ has the eigenvalues $0$, and $\frac{8}{3}$. Therefore, the complete list of eigenvalues of $Q_{(B,H)}$ are $0$ with multiplicity $5$, $\frac{4}{3}$ with multiplicity $2$, and $\frac{8}{3}$ with multiplicity $1$. 
 
 For this hypergraph $H$, and the same choices of $\delta_{V(H)}$, and $\delta_{E(H)}$ the Laplacian matrix $L_{(B,H)}=\frac{1}{3}\left[\begin{smallmatrix}
   \phantom{-} 6 & -2 & -1 & -1 &   \phantom{-}0 &   \phantom{-}0 & -1 & -1\\ -2 &   \phantom{-}6 & -1 & -1 &   \phantom{-}0 &   \phantom{-}0 & -1 & -1\\ -1 & -1 &   \phantom{-}6 & -2 & -1 & -1 &   \phantom{-}0 &   \phantom{-}0\\ -1 & -1 & -2 &   \phantom{-}6 & -1 & -1 &   \phantom{-}0 &   \phantom{-}0\\  \phantom{-} 0 &   \phantom{-}0 & -1 & -1 &   \phantom{-}6 & -2 & -1 & -1\\   \phantom{-}0 &   \phantom{-}0 & -1 & -1 & -2 &   \phantom{-}6 & -1 & -1\\ -1 & -1 &   \phantom{-}0 &   \phantom{-}0 & -1 & -1 &   \phantom{-}6 & -2\\ -1 & -1 &   \phantom{-}0 &   \phantom{-}0 & -1 & -1 & -2 &   \phantom{-}6
 \end{smallmatrix}\right]$. Here, $ L_{(B,\hat H)}=\left[\begin{smallmatrix}
   \phantom{-}  1&-1\\-1&\phantom{-}1
 \end{smallmatrix}\right]$, and the eigenvalues of this matrix are $0$, and $2$. Since $k\frac{cc_E}{c_V}=\frac{4}{3}$, by the \Cref{lap_contract}, two eigenvalues of $L_{(B,H)}$ are $0$, and $\frac{8}{3}$. For the same hypergraph, and the same choices of $\delta_{V(H)}$, and $\delta_{E(H)}$, the adjacency matrix is $A_{(B,H)}=\frac{1}{3}\left[
 \begin{smallmatrix}
     0 & 2 & 1 & 1 & 0 & 0 & 1 & 1\\ 2 & 0 & 1 & 1 & 0 & 0 & 1 & 1\\ 1 & 1 & 0 & 2 & 1 & 1 & 0 & 0\\ 1 & 1 & 2 & 0 & 1 & 1 & 0 & 0\\ 0 & 0 & 1 & 1 & 0 & 2 & 1 & 1\\ 0 & 0 & 1 & 1 & 2 & 0 & 1 & 1\\ 1 & 1 & 0 & 0 & 1 & 1 & 0 & 2\\ 1 & 1 & 0 & 0 & 1 & 1 & 2 & 0 
 \end{smallmatrix}
 \right]$. For this matrix, the matrix $C_H=\frac{2}{3}\left[\begin{smallmatrix}
     1&2\\2&1
 \end{smallmatrix}\right]$. Since $2$, and $-\frac{2}{3}$ are two eigenvalues of $C_H$, by the \Cref{twin-eqp-eig}, $2$, and $-\frac{2}{3}$ are also two eigenvalues of $A_{(B,H)}$.
 In the next section, we show more application of the results presented in this section while computing the spectra of hyperflowers.
\subsection{Complete spectra of hyperflowers}
 \label{sec-hypflow}
Here, we compute spectra of a class of hypergraphs, \textit{$(l,r)$-hyperflower with $t$-twins and $m$-homogeneous centre}, using our previous results.
An $(l,r)$-hyperflower with $t$-twins and $m$-homogeneous centre is a hypergraph $H$ where $V(H)$ can be expressed as the disjoint partition $V(H)=U\cup W$ with $U\cap W=\emptyset$. Both the set $U$, and $W$ are partitioned as $U=\bigcup\limits_{i=1}^l U_i$ with $|U_i|=t$, $U_i=\{u_{is}\}_{s=1}^t$, $U_i\cap U_j=\emptyset$ for all $i\neq j$ for all $i,j=1,2\ldots,l$, and  $W=\bigcup\limits_{p=1}^r e_p$ with $|e_p|=m$, and $e_p\cap e_q=\emptyset$ for all $p(\ne q)$, $p,q=1,2,\ldots,r$. We refer to each $U_i$ as a peripheral component and $U$ as the periphery of the hyperflower. Each $e_k$ is called a central component, and $W$ is called the centre of the hyperflower. The hyperedge set is $E(H)=\{e_{ki}:e_{ki}=e_k\cup U_i,k=1,2,\ldots,r;i=1,2,\ldots,l\}$.


\subsection*{Hyperflower to spectra}
Any $(l,r)$-hyperflower contains only two $\mathfrak{R}_{tw}$-equivalence classes. One of them contains $l$ peripheral components, and the other one contains $r$ central components. Each one of the peripheral and central components is a unit. Therefore, by using \Cref{cor_unit_eig}, \Cref{com-unil-l}, \Cref{com-unil-a}, \Cref{adj_twin_unit}, and \Cref{lap_twin_unit}, we have the following result.

\begin{thm}\label{list-hyp-eig}  Let $H$ be a $(l,r)$-hyperflower with $t$-twins and $m$-homogeneous center with periphery $U$, center $W$, peripheral components $U_i$, for $i=1,\ldots,l$ and layers of the center $e_j$, for $j=1,\ldots,r$.
\begin{enumerate}[leftmargin=*]
	\item For each $U_i$, one eigenvalue of $Q_{(H,\delta_{V(H)},\delta_{E(H)})}$ is $0$ with multiplicity $|U_i|-1$ and if $\delta_{V(H)}(v)=c_a^i$ for all $v\in U_i $,
	then $-\frac{1}{c_a^i}\sum\limits_{e\in E_{U_i} }\frac{\delta_{E(H)}(e)}{|e|^2}$ and $\frac{1}{c_a^i}\sum\limits_{e\in E(H)_{U_i} }\frac{\delta_{E(H)}(e)}{|e|}$ are eigenvalues of $A_{(H,\delta_{V(H)},\delta_{E(H)})}$ and $L_{(H,\delta_{V(H)},\delta_{E(H)})}$ respectively, with multiplicity at least $|U_i|-1$ for all $i=1,2,\ldots,l$. 
	\item For each $e_j$, one eigenvalue of $Q_{(H,\delta_{V(H)},\delta_{E(H)})}$ is $0$ with multiplicity $|e_j|-1$ and if $\delta_{V(H)}(v)=c_b^j$ for each $v\in e_j $,
	then  $-\frac{1}{c_b^j}\sum\limits_{e\in E_{e_j} }\frac{\delta_{E(H)}(e)}{|e|^2}$ and $\frac{1}{c_b^j}\sum\limits_{e\in E(H)_{e_j} }\frac{\delta_{E(H)}(e)}{|e|}$ are eigenvalues of $A_{(H,\delta_{V(H)},\delta_{E(H)})}$ and $L_{(H,\delta_{V(H)},\delta_{E(H)})}$ respectively, with multiplicity at least $|e_j|-1$ for all $e_j$, $j\in [r]$. 
	\item $\frac{w}{c_a}r(t-1)$, $\frac{w}{c_a}rm$, and  $\frac{w}{c_a}rt$ are eigenvalues of $A_{(H,\delta_{V(H)},\delta_{E(H)})}$, $L_{(H,\delta_{V(H)},\delta_{E(H)})}$, and $Q_{(H,\delta_{V(H)},\delta_{E(H)})}$ respectively, with multiplicity at least $l-1$.
	\item $\frac{w}{c_b}l(m-1)$, $\frac{w}{c_b}lt$, and $\frac{w}{c_b}lm$ are eigenvalues of $A_{(H,\delta_{V(H)},\delta_{E(H)})}$, $L_{(H,\delta_{V(H)},\delta_{E(H)})}$ and $Q_{(H,\delta_{V(H)},\delta_{E(H)})}$ respectively, with multiplicity at least $r-1$.
\end{enumerate}
\end{thm}
\begin{proof}
   Both $U_i$, and $e_j$ are unit for $i=1,\ldots,l$, and $j=1,\ldots,r$. Consequently, the conditions $\delta_{V(H)}(v)=c_a^i$ for all $v\in U_i $, and $\delta_{V(H)}(v)=c_b^j$ for each $v\in e_j $ implies that $U_i$, and $e_j$ corresponds to respectively, $|U_i|$, and $|e_i|$ numbers of same rows in the matrix representation of $\q$. Therefore, as indicated by the \Cref{cor_unit_eig},  $|U_i|$, and $|e_i|$ , respectively, leads to $0$ eigenvalue of $\q$ with multiplicity $|U_i|-1$, and $|e_i|-1$. Similarly,  parts (1) and (2) of the theorem follow from \Cref{com-unil-l}, \Cref{com-unil-a}.

  The collection of  all the  peripheral components $\{U_i:i=1,\ldots,l\}$ is an $\mathfrak{R}_{tw}$-equivalence class with $|U_i|=t$, and each $|U_i|$ is a unit $U=W_E$ with $|E|=r$. that is the generator size of this $\mathfrak{R}_{tw}$-equivalence class is $r$. Similarly, the collection of all the central components $\{e_1,\ldots,e_r\}$ is an $\mathfrak{R}_{tw}$-equivalence class with $|e_j|=m$, and the generator size of this equivalence class is $l$. Therefore, the part (3) and (4) of this Theorem follow from the \Cref{adj_twin_unit}, and \Cref{lap_twin_unit}.
\end{proof}
It is easy to verify that $H$ is a uniform hypergraph with $\hat H=K_2$, a complete graph with $2$ vertices if and only if $H$ is a $(l,r)$-hyperflower with $t$-twins and $m$-homogeneous centre for some $l,r,t,m\in \mathbb{N}$. The \Cref{list-hyp-eig}, provide all the eigenvalues from the units and twin units, except two. The next Theorem gives us these remaining two eigenvalues for $A_{(H,\delta_{V(H)},\delta_{E(H)})}$, $L_{(H,\delta_{V(H)},\delta_{E(H)})}$, and $Q_{(H,\delta_{V(H)},\delta_{E(H)})}$.
\begin{thm}\label{lastonelap}
Suppose that $H$ is a $(l,r)$-hyperflower with $t$-twins and an $m$-homogeneous center. If $i)\frac{\delta_{E(H)}(e)}{|e|^2}=w$ for all $e\in E(H)$ for some constant $w$ and $ii)\delta_{V(H)}(v)=c$ for all $v\in V(H)$ for some constant $c$ then
\begin{enumerate}[leftmargin=*]
	\item $\frac{w}{c}r(t-1+m\gamma)$ is an eigenvalue of $A_{(H,\delta_{V(H)},\delta_{E(H)})}$ with eigenvector $\gamma\chi_W+\chi_U$,
	where $\gamma$ is a root of the equation
	\begin{equation} \label{hyp-root-eigen}
		rmx^2+[r(t-1)-l(m-1)]x-lt=0.
	\end{equation}
	\item $\frac{w}{c}|V(H)|$ is an eigenvalue of $L_{(H,\delta_{V(H)},\delta_{E(H)})}$ with eigenvector $|W|\chi_{U}-|U|\chi_{W} $.
	\item $\frac{w}{c} (lm+rt)$ and $0$ are two eigenvalues of $Q_{(H,\delta_{V(H)},\delta_{E(H)})}$ with eigenvectors $y=l\chi_{W}+r\chi_{U}$ and $z=\sum\limits_{i=1}^l\chi_{\{u_{i1}\}}-\sum\limits_{j=1}^r\chi_{\{v_{j1}\}}$ respectively, for some fixed $u_{i1}\in U_i$, $ v_{j1}\in e_j$, $i=1,\ldots,l$, and $j=1,\ldots, r$. 
\end{enumerate}
\end{thm}
\begin{proof}
Suppose that $y_{\gamma}=\gamma\chi_W+\chi_U$. Therefore, 
\begin{align*}
	(A_{(H,\delta_{V(H)},\delta_{E(H)})}y_{\gamma})(v)
	&=
	\begin{cases}
		r\frac{w}{c}(t-1+m\gamma)  &\text{~if~} v\in U,\\
		l\frac{w}{c}(t+\gamma(m-1))
		&\text{~otherwise.}
	\end{cases}.
\end{align*}
Therefore, if $y_{\gamma}$ is an eigenvector of $A_{(H,\delta_{V(H)},\delta_{E(H)})}$ then $\gamma$ is a root of the equation 
$$rmx^2+[r(t-1)-l(m-1)]x-lt=0,$$
and in that case $r\frac{w}{c}(t-1+m\gamma)$ is an eigenvalue of $A_{(H,\delta_{V(H)},\delta_{E(H)})}$.

Consider $y=|W|\chi_{U}-|U|\chi_{W}\in \mathbb{C}^V$. 
Note that for any $v\in U_i$, one has $E_v(H)=E_{U_i}$ for all $i\in [l]$. Since, the center of the hyperflower is $m$-uniform and $|E_{U_i}|=r$ for all $i\in [l]$, one has $\sum\limits_{u\in e}(y(v)-y(u))=m(|W|+|U|)$ for any $e\in E_{U_i}$. Therefore, by using the facts $\frac{\delta_{E(H)}(e)}{|e|^2}=w$ for all $e\in E(H)$ for some constant $w$ and $\delta_{V(H)}(v)=c$ for all $v\in V(H)$ for some constant $c$, for any $v\in U_i$, we have if $v\in U$ then $(Ly)(v)=\frac{w}{c}(|U|+|W|)mr=\frac{w}{c}(|U|+|W|)y(v)$.

Similarly, using the facts, $|U_i|=t$ for all $i\in [l]$ and $|E_{e_j}|=l$ for all $j\in [r]$, we have  if $v\in W$ then $(Ly)(v)=\frac{w}{c}(|U|+|W|)lt=\frac{w}{c}(|U|+|W|)y(v)$. The result on $Q_{(H,\delta_{V(H)},\delta_{E(H)})}$ follows from the fact that $\sum\limits_{u\in e}y(u)=(lm+rt)$ and $\sum\limits_{u\in e}z(u)=0$ for all $e\in E(H)$. This completes the proof.
\end{proof}

\subsection*{Spectra to hyperflower}
We can reconstruct the hyperflower from its spectra( of $Q_{(H,\delta_{V(H)},\delta_{E(H)})}$, $A_{(H,\delta_{V(H)},\delta_{E(H)})}$ and $L_{(H,\delta_{V(H)},\delta_{E(H)})}$). Since we know $H$ is a hyperflower, we need to either figure out the peripheral and central components or compute the value of $l,r,t,m$.

Suppose we know the spectra of $Q_{(H,\delta_{V(H)},\delta_{E(H)})}$ of a hyperflower, consider the eigenvector $y$ corresponding to the largest eigenvalue. The range of $y$ consists of only two real numbers $a,b$. One of the $\{y^{-1}(a), y^{-1}(b)\}$ is the center $W$ and the another is the periphery $U$ of the hyperflower. For each maximal subset, $U_i$ of $U$ with $S_{U_i}$ is a subspace of the eigenspace of $0$ we have $U_i$ as a peripheral component. Similarly, for each maximal subset $e_k$ of $W$ with $S_{e_k}$ as a subspace of the eigenvalue $0$, we have $e_k$ as a central component.    
Similarly, considering the eigenvector corresponding to the largest eigenvalue of $A_{(H,\delta_{V(H)},\delta_{E(H)})}$, we can separate the periphery and centre of $H$. Each maximal subset $U_i$ of the periphery with $S_{U_i}$ is a subspace of the eigenspace of the eigenvalue of $A_{(H,\delta_{V(H)},\delta_{E(H)})}$, and similarly, we can point out the central components.
In the same way, we can reconstruct a hyperflower from the spectra of $L_{(H,\delta_{V(H)},\delta_{E(H)})}$.
	
	Since a hyperflower $H$ is such a hypergraph that $\hat H=K_2$, other than $2$ eigenvalues associated with the quotient matrices related to the contraction $\hat{H}$, all the eigenvalues of $\lap,\adj$, and $\q$ can be calculated from units and twin units. Thus, using our results in the previous sections, we can provide the complete spectra of hyperflowers. If for a hypergraph $H$, the contraction hypergraph $\hat{H}$ contains many vertices, then using our results, we can provide only a partial list of eigenvalues from units and twin units, and for the remaining eigenvalues, we need to calculate the eigenvalues of the quotient matrix associated with $\hat{H}$ provided in the previous sections. 

\section{Some Applications}\label{app}
\subsection{Co-spectral hypergraphs}From \Cref{cor_unit_eig}, we can conclude $H$, and $H'$ are two hypergraphs such that there exists a cardinality preserving bijection $\mathfrak{g}:\mathfrak{U}(H)\to \mathfrak{U}(H')$ with $\hat{Q}_H$, and $\hat{Q}_{H'}$ are co-spectral then the general signless Laplacians of the two hypergraphs are co-spectral.
 \begin{prop}\label{cospec-q}
     Let $H$ and $H'$ be two hypergraphs with $|V(H)|=|V(H')|$.  If there exists a bijection $\mathfrak{g}:\mathfrak{U}(H)\to \mathfrak{U}(H')$ with $|\mathfrak{g}(W_E)|=|W_E|$, and $\hat{Q}_H$, and $\hat{Q}_{H'}$ are co-spectral, then $Q_{(H,\delta_{V(H)},\delta_{E(H)})}$ and $Q_{(H',\delta_{V(H')},\delta_{E(H')})}$ are co-spectral. 
     
 \end{prop}
 \begin{proof}
     Since there exists a bijection $\mathfrak{g}:\mathfrak{U}(H)\to \mathfrak{U}(H')$, such that $|\mathfrak{g}(W_E)|=|W_E|$, all the eigenvalues described in \Cref{cor_unit_eig}(1) for $Q_{(H',\delta_{V(H')},\delta_{E(H')})}$ are equal. Since $\hat{Q}_H$, and $\hat{Q}_{H'}$ are co-spectral, all the eigenvalues provided in \Cref{cor_unit_eig}(2) are same for $Q_{(H,\delta_{V(H)},\delta_{E(H)})}$ and $Q_{(H',\delta_{V(H')},\delta_{E(H')})}$. Thus, these two operators are also co-spectral, and the result follows.
 \end{proof}
Using the following example, we describe the results in \Cref{cospec-q}.
\begin{exm}
Consider two hypergraphs $H, H'$ with $V(H)=V(H')=\{1,2,\ldots,8\}$, $E(H)=\{\{1, 2, 3, 4\},$  $\{1,2,5,6\},\{3,4,5,6\},\}$ and, $E(H')=\{\{1, 2, 7, 8\},\{3,4,7,8\},\{5,6,7,8\}\}$. Suppose that $\delta_{V(H)}$, $\delta_{V(H')}$, $\delta_{E(H)}$, $\delta_{E(H')}$ are chosen as in \cite{MR4208993}. The matrix representation of $\hat{Q}_H$ and $\hat{Q}_{H'}$  are $\hat{Q}_H=\frac{2}{3}\left[\begin{smallmatrix}
    2& 1& 1& 0\\
    1& 2& 1& 0\\
    1& 1& 2& 0\\
    0&0&0&0
\end{smallmatrix}\right]$, and 
$\hat{Q}_{H'} =
\frac{2}{3}\left[\begin{smallmatrix}
  1&0&0&1\\
  0&1&0&1\\
  0&0&1&1\\
  1&1&1&3
\end{smallmatrix}\right]$, respectively. Since $\hat{Q}_H$ and $\hat{Q}_{H'}$ are co-spectral, $Q_{(H,\delta_{V(H)},\delta_{E(H)})}$ and $Q_{(H',\delta_{V(H')},\delta_{E(H')})}$ are also co-spectral. 
 \end{exm}
\subsection{Random-walk on hypergraphs using units}\label{sec-unit-rand-walk}
\textit{Random walk on graphs} (\cite{introrand}) is an active area of study for a long time( \cite{grimmett2018probability}). In the last decade, random walks on hypergraphs have attracted considerable attention \cite{MR4420488,MR3980498,hg-mat}. A random walk on hypergraph $H$ is a map $\mathfrak{r}_H:\mathbf{T}\to V(H)$, such that $\mathbf{T}=\mathbb{N}\cup\{0\}$,
and $\mathfrak{r}_H(t)$ depends only on its previous state $\mathfrak{r}_H(t-1)$ for all $t(\ne 0)$.
Thus, $\mathfrak{r}_H $ depends on a matrix $P_H$ of order $|V(H)|$, referred as the probability transition marix, defined as $P_H(u,v)=Prob(\mathfrak{r}_H(i+1)=v|\mathfrak{r}_H(i)=u)$, 
and $\sum\limits_{u\in V(H)}P_H(u,v)=1$. Suppose that $x_i(\in[0,1]^{V(H)})$ is the probability distribution at time $i$, that is, $x_i(v)=Prob(\mathfrak{r}_H(i)=v)$ and the iteration $x_{i+1}=P^T_Hx_i$ defines the random walk.

Here, we consider the transition probability matrix $P_H$ for random walk studied in \cite{up2021} and is defined as 
$${P_H}(u,v)=
\begin{cases}
\frac{1}{r(u)}\sum\limits_{e\in E_u(H)\cap E_v(H)}\frac{\delta_{E(H)}(e)}{\delta_{V(H)}(u)}\frac{1}{|e|^2} & \text{~if~} u\neq v,\\
0& \text{~otherwise,~}
\end{cases} $$
for all $u,v\in V(H)$, where $r(u):=\sum\limits_{e\in E_u(H)}\frac{\delta_{E(H)}(e)}{\delta_{V(H)}(u)}\frac{|e|-1}{|e|^2}$. 

Let $\Delta_H=I-P_H$. For any $x\in \mathbb{C}^{V(H)}$, and $v\in V(H)$ we have \begin{align*}
(\Delta_H x)(u)
&=x(u)-(P_H x)(u)
\\&=\sum\limits_{e\in E_v(H)}\frac{\delta_{E(H)}(e)}{\delta_{V(H)}(u)r(u)}\frac{1}{|e|^2}\sum\limits_{u\in e }(x(u)-x(v))
\end{align*}

Let $\mathfrak{r}_H$ be a random walk on a hypergraph $H$. For any $v\in V(H)$ and $t\in\n$, we denote the event $\{\mathfrak{r}_H(t)=v,\mathfrak{r}_H(i)\ne v\text{~for~}0\le i<t\}$ of  \textit{first hitting the vertex $v$ at time $t$} as $(T_v=t )$. The event $\{\mathfrak{r}_H(t)=v,\mathfrak{r}_H(i)\ne v\text{~for~}0\le i<k|\mathfrak{r}_H(0)=u\}$ is referred as the \textit{ first hitting of $v$ at time $t$ where $\mathfrak{r}_H$ starts at $u$} and is denoted by $(T^u_v=t)$. The expected time taken by the random walk to reach from $u$ to $v$ is referred to as \textit{expected hitting time from $u$ to $v$} and is defined as $ET(u,v)=\sum\limits_{k\in\n}kProb(T^u_v=k)$.
\begin{thm}\label{unit-hit}
Let $\mathfrak{r}_H$ be a random walk on a hypergraph $H$. If $W_{E_0}\in \mathfrak{U}(H)$ then for any $t\in \n$, we have $Prob(T^u_{u_1}=t)=Prob(T^u_{u_2}=t)$ for any $u_1,u_2\in W_{E_0}$ and $u\in V(H)\setminus\{u_1,u_2\}$.
\end{thm}
\begin{proof}
For any $u,v(\ne u)\in V(H)$, and $t\in\n$, we have $(T^u_v=t)=A_t\cap A_{t-1}\cap\ldots\cap A_0$, where $A_t=(\mathfrak{r}_H(t)=v)$, $A_i=(\mathfrak{r}_H(i)\ne v)$ for all $i=1,\ldots,t-1$, $A_0=(\mathfrak{r}_H(0)=u)$. Therefore,
$$  Prob(T^u_v=t)=(\sum\limits_{\mathfrak{r}_H(t-1)(\ne v)\in V(H)}P_H(\mathfrak{r}_H(t-1),v))\prod\limits_{i=1}^{t-1}(\sum\limits_{\mathfrak{r}_H(t-2)(\ne v)\in V(H)}P_H(\mathfrak{r}_H(t-i-1),\mathfrak{r}_H(t-i))).$$
For any $u_1,u_2\in W_{E_0}$, since $P_H(u,u_1)=P_H(u,u_2)$ for all $u\in V(H)\setminus\{u_1,u_2\}$ and there exists a bijection between $(T^u_{v_1}=t),(T^u_{v_2}=t)$ defined by $(v_0=u,v_1,\ldots,v_{t-1},v_t=u_1)\mapsto (v_0=u,v_1,\ldots,v_{t-1},v_t=u_2)$, we have $Prob(T^u_{u_1}=t)=Prob(T^u_{u_2}=t)$ for any $u_1,u_2\in W_{E_0}$ and $u\in V(H)\setminus\{u_1,u_2\}$.
\end{proof}
Similarly, we can show $Prob(T_u^{u_1}=t)=Prob(T_u^{u_2}=t)$ for any $u_1,u_2\in W_{E_0}$ and $u\in V(H)\setminus\{u_1,u_2\}$.
\begin{cor}
Let $\mathfrak{r}_H$ be a random walk on a hypergraph $H$. If $W_{E_0}\in\mathfrak{U}(H)$ and $u_1,u_2\in W_{E_0}$ then $ET(u,u_1)=ET(u,u_2)$ and $ET(u_1,u)=ET(u_2,u)$ for all $u\in V(H)\setminus\{u_1,u_2\}$.
\end{cor}
We can show a similar result for symmetric sets.
\begin{thm}\label{sym-hit}
Let $\mathfrak{r}_H$ be a random walk on a hypergraph $H$. If $U$ is a $\sigma_H$-symmetric set with $\delta^{\prime}_{V(H)}(v)=c$, a constant for all $v\in U$ then for any $t\in \n$, we have $Prob(T^u_{u_1}=t)=Prob(T^u_{u_2}=t)$ for any $u_1,u_2\in U$ and $u\in V(H)\setminus\{u_1,u_2\}$.
\end{thm}
\begin{proof}
Since $U$ is a $\sigma_H$-symmetric set with $\delta^{\prime}_{V(H)}(v)=c$, a constant for all $v\in U$ we have $P_H(u_1,u)=P_H(u_2,u)$ for all $u\in V(H)\setminus\{u_1,u_2\}$. Thus, by proceeding like the previous proof, the result follows.
\end{proof}
Suppose that $W_{E_i}$ and $W_{E_j}$ are twin units. If there exists a bijection $\mathfrak{h_{ij}}:W_{E_i}\to W_{E_j}$ and $x\in\mathbb{C}^{V(H)}$ such that $x(\mathfrak{h}_{ij}(v))=x(v)$ for all $v\in W_{E_i}$ then $\mathfrak{h_{ij}}:W_{E_i}\to W_{E_j}$ is called $x$-preserving. Similarly, the canonical bijection $\mathfrak{f}_{E_iE_j}:E_i\to E_j$ is called $\alpha$-preserving for some $\alpha(\in\mathbb{C}^{E(H)})$ if $\alpha(\mathfrak{f}_{E_iE_j}(e))=\alpha(e)$ for all $e\in E_i$.
\begin{thm}
Let $\mathfrak{r}_H$ be a random walk on a hypergraph $H$. If $W_{E_i}$ and $W_{E_j}$ are twin unit and $\mathfrak{h_{ij}}:W_{E_i}\to W_{E_j}$ is a $\delta_{V(H)}$-preserving bijection with $\mathfrak{h_{ij}}(u_1)=u_2$ and the canonical bijection  $\mathfrak{f}_{E_iE_j}:E_i\to E_j$ is $\delta_{E(H)}$-preserving then for any $t\in \n$, we have $Prob(T^u_{u_1}=t)=Prob(T^u_{u_2}=t)$ for any $u_1\in W_{E_i}$ and $u_2\in W_{E_j}$ and $u\in V(H)\setminus(W_{E_i}\cup W_{E_j})$.
\end{thm}
\begin{proof}
For each event $F_1=\{\mathfrak{r}_H(t)=u_1|\mathfrak{r}_H(i)=v_i,1\le i<t, \mathfrak{r}_H(0)=u\}$, there exist an event $F_1=\{\mathfrak{r}_H(t)=u_2|\mathfrak{r}_H(i)=v^\prime_i,1\le i<t, \mathfrak{r}_H(0)=u\}$ such that
$$v^\prime_k=\begin{cases}
	v_k & \text{~if~}v_k\in V(H)\setminus(W_{E_i}\cup W_{E_j}),\\
	h_{ij}(v_k) &\text{~if~} v_k\in W_{E_i},\\
	h^{-1}_{ij}(v_k) &\text{~if~} v_k\in W_{E_j}. 
\end{cases}$$
Since $h_{ij}$  is a $\delta_{V(H)}$-preserving bijection and the canonical bijection  $\mathfrak{f}_{E_iE_j}$ is $\delta_{E(H)}$-preserving, we have $Prob(F_1)=Prob(F_2)$. Therefore, $Prob(T^u_{u_1}=t)=Prob(T^u_{u_2}=t)$ for any $u_1\in W_{E_i}$ and $u_2\in W_{E_j}$ and $u\in V(H)\setminus(W_{E_i}\cup W_{E_j})$.
\end{proof}
\subsection{Unit-distance: A pseudo metric}\label{sec-pseudomet}
In a hypergraph $H$, a path between $v_0,v_l(\in V(H))$ of length $l$ is an alternating sequence of distinct vertices and edges $v_0e_1v_1\ldots v_{l-1}e_lv_l$ with $v_{i-1},v_i\in e_i$ for all $i=1,\ldots,l$ (see \cite{hg-mat}). In the case of a walk between $v_0,v_l$, the vertices and hyperedges may not be distinct.  This notion of the path and walk naturally induces a metric defined by $d(u,v)=$ length of the smallest path between $u,v(\in V(H))$. This notion of paths and distance has some limitations. 

\begin{figure}[ht]
\centering
\begin{tikzpicture}[scale=0.7]
	
		\node [style=none] (2) at (-11, -0.25) {};
		\node [style=none] (3) at (-3, -0.25) {};
		\node [style=none] (4) at (-8.5, -0.25) {};
		\node [style=none] (5) at (2.5, -0.25) {};
		\node [style=new style 0] (6) at (-9.5, -0.25) {\tiny 1};
		\node [style=new style 0] (7) at (-7, -0.25) {\tiny 2};
		\node [style=new style 0] (8) at (-4.5, -0.25) {\tiny 3};
		\node [style=none] (9) at (-2.25, -0.25) {};
		\node [style=none] (10) at (5.75, -0.25) {};
		\node [style=new style 0] (11) at (-1, -0.25) {\tiny 4};
		\node [style=new style 0] (12) at (1.25, -0.25) {\tiny 5};
		\node [style=new style 0] (13) at (4, -0.25) {\tiny 6};
		\node [style=none] (14) at (-10.25, 1) {\tiny e};
		\node [style=none] (15) at (-3.25, 1.25) {\tiny f};
		\node [style=none] (16) at (4.75, 1) {\tiny g};
	
		\draw (3.center)
		to [bend left, looseness=0.75] (2.center)
		to [bend left, looseness=0.75] cycle;
		\draw (5.center)
		to [bend left, looseness=0.75] (4.center)
		to [bend left, looseness=0.75] cycle;
		\draw (10.center)
		to [bend left, looseness=0.75] (9.center)
		to [bend left, looseness=0.75] cycle;
	
\end{tikzpicture}

\caption{A path of length $3$ between $1$ and $6$.}
\label{fig:walk-hyp}
\end{figure}
Consider the path of length $3$ between $1$ and $6$ in \Cref{fig:walk-hyp} composed of the hyperedges $e,f,g$. This path has multiple representations, and they are $1e2f4g6$, $1e3f4g6$, $1e2f5g6$, and $1e3f5g6$. Thus, this notion of the path leads us to an unavoidable redundancy that may count the same path multiple times. This redundancy will affect any notion that depends on the number of paths (or walks) between two vertices. For example, for a graph $G$, the $(u,v)$-th element of the $k$-th power of adjacency is the number of walk of length $k$ between $u,v\in V(G)$. This result does not hold for hypergraphs.

A vertex $u$ should be closer to a vertex belonging to the same unit than two other vertices outside the unit. This is not taken care of by the usual distance. For example, in the path given in \Cref{fig:walk-hyp}, we have $d(2,3)=1=d(2,4)$ though $3,2$ belongs to the same unit and $2,4$ belong to different units. 
Here we construct a pseudo metric space on the set of vertices using the units. Suppose that $H$ is a hypergraph. Consider the graph $G_H$, defined by $V(G_H)=\mathfrak{U}(H)$ and $$E(G_H)=\{\{W_{E_i},W_{E_j}\}\in {\mathfrak{U}(H)}^2:W_{E_i}\text{~and~}W_{E_j} \text{~are neighbours~}\}.$$

For any hypergraph $H$, an \textit{unit-walk}, and \textit{unit-path} between $u,v$ of length $l$ are respectively the walk and path of length $l$ between $W_{E_u(H)},W_{E_v(H)}$ in $G_H$ and  the \textit{unit-distance} $d^U_H:V(H)\times V(H)\to [0,\infty)$ is defined by $d^U_H(u,v)=$length of the smallest path between $W_{E_u(H)},W_{E_v(H)}$ in $G_H$. Here $d^U_H(u,v)=0$ if $u,v$ belongs to the same unit.

For any graph $G$, the distance $d_G:V(G)\times V(G)\to[0,\infty)$ is defined as $d_G(u,v)=$length of smallest path between $u,v(\in V(G))$.
Evidently, for any hypergraph $H$, we have $d^u_H(u,v)=d_{G_H}(W_{E_u(H)}, W_{E_v(H)})$ for all  $u,v\in V(H)$.
We now generalize some notions of distance-based graph-centrality for the hypergraph using the notion of unit distance. For any connected hypergraph $H$, and $v\in V(H)$ the \textit{unit-ecentricity} $\xi_H\in \mathbb{C}^{V(H)}$ defined by $\xi_H(v)=\max\limits_{u\in V(H)}d^U_H(u,v)$. The \textit{unit-diameter} of $H$ is $\max\limits_{v\in V(H)}\xi_H(v)$ and the \textit{unit-radius} is $\min\limits_{v\in V(H)}\xi_H(v)$. Thus, for any hypergraph $H$, the number of hyperedges is greater than or equal to the unit-diameter of $H$.
For a hypergraph $H$, we define the \textit{unit-girth} of $H$ as the girth of $G_H$.
For each $e\in E(H)$, there exists $n_e\in \n$ such that $e=\bigcup\limits_{i=1}^{n_e}W_{E_i}$ where $W_{E_i}\in \mathfrak{U}(H)$ for all $i=1,\ldots, n_e$. The \textit{maximum partition number} of $H$ is $\max\limits_{e\in E(H)}n_e$ and \textit{minimum partition number} of $H$ is $\min\limits_{e\in E(H)}n_e$. Thus, if the minimum partition number of a hypergraph $H$ is at least $3$ then
the minimum partition number of $H$ is at least equal to the unit-grith of $H$, which is equal to $3$.
The \textit{unit-clique number} of a hypergraph is the clique number of $G_H$.
\begin{prop}
	The maximum partition number of a hypergraph $H$ can be, at most, the unit-clique number of $H$.
\end{prop}	
\begin{proof}Let $H$ be a hypergraph and $c_H$ be the unit-clique number of $H$.
	It is enough to prove $n_e\le c_H$ for all $e\in E(H)$. If possible, let there exists $e\in E(H)$ such that $n_e>c_H$. Thus, $e=\bigcup\limits_{i=1}^{n_e}W_{E_i}$ and this leads us to a $n_e$-clique in $G_H$ with $n_e> c_H$, a contradiction. Therefore,  $n_e\le c_H$ for all $e\in E(H)$.
\end{proof}

For any graph $G$, consider the usual adjacency matrix $\tilde A_G=\left(a_{uv}\right )_{u,v\in V(G)}$ defined by 
$$a_{uv}=\begin{cases}
	1 &\text{~if~} \{u,v\}\in E(G) ,\\
	0 &\text{~otherwise.}
\end{cases}$$
Note that if $\delta_{E(G)}(e)=\frac{|e|^2}{|e|-1}$ for all $e\in E(G)$ and $\delta_{V(G)}(v)=1$ for all $v\in V(G)$ then the matrix of the general adjacency operator with respect to usual basis becomes this usual adjacency matrix, that is for choice mentioned above of $\delta_{E(H)},\delta_{V(H)}$, we have $A_Gx=\tilde A_Gx$, for all $x\in \mathbb{C}^{V(G)}$.
Since in a graph $G$, the $(u, v)$-th entry of $\tilde A_G^k$ is the number of walk of length $k$ between $u,v\in V(G)$. For any hypergraph $H$, and $u,v\in V(H)$, the $(u,v)$-th entry of $\tilde A^k_{G_H}$ is the number of unit-walk of length $k$ between $u,v$.
Using the concept of unit distance, we can incorporate some distance-based graph centralities into the hypergraph theory. For instance, in a connected hypergraph $H$, the \textit{unit-closeness centrality} of each $v(\in V(H))$ is defined by $c_U(v)=\frac{1}{\sum\limits_{u\in V(H)}d^U_H(u,v) }$, the \textit{unit-eccentricity based centrality} of $v(\in V(H))$ is defined as $c_E(v)=\frac{1}{\xi_H(v)}$.

\subsection{Proper colouring, independent set and contraction of a hypergraph}\label{hyp-color}
A hypergraph $H $ is called a simple hypergraph if for any two $e,f\in\mathfrak{U}(H)$, $e\subseteq f$ if and only if $e=f$. A \textit{proper colouring} of a hypergraph $H$ is a vertex colouring such that no $e\in E(H)$ 
is monochromatic. If a proper colouring of a hypergraph $H$ is possible with $n$-colours, then $H$ is called \textit{$n$-colourable hypergraph}. The \textit{chromatic number} of $H$, $\chi(H)$ is the minimum number of colours required for proper colouring of $H$.
\begin{prop}\label{col-contract}
	Let $H$ be a simple connected hypergraph $H$ with $|E(H)|>1$. If $\hat H$ is $n$-colourable then $H$ is $n$-colourable.
\end{prop}
\begin{proof}
	Let $\hat H$ be $n$-colourable and $V_i$ be the set of vertices of $\hat H$, coloured by the $i$-th colour  
	for all $i=1,\ldots, n$.     
	Since 
	$e\not\subseteq V_i$ for any $e\in E(\hat H)$, thus, $e^\prime\not\subseteq\pi^{-1}(V_i)$ for all $e^\prime \in E(H)$. Thus, we can colour all the vertices in $\pi^{-1}(V_i)$ with the $i $-th colour, for all $i=1,\ldots, n$ and therefore, $H$ is $n$-colourable.
\end{proof}
Now, we have the following result.
\begin{cor}
	For any simple connected hypergraph $H$ with $|E(H)|>1$, $\chi(H)\le\chi(\hat H)\le |\mathfrak{C}(\mathfrak{U}(H))|$.
\end{cor}
The equality of the above Corollary holds for many hypergraphs; for example, the chromatic number of hyperflower, blow-up of any bipartite graph (in which a unit replaces each vertex of the bipartite graph) are equal to $\chi(H)=| \mathfrak{C}(\mathfrak{U}(H))|=2$. For the power graph of any bipertite graph, $ | \mathfrak{C}(\mathfrak{U}(H))|=3$ but $\chi(H)=2 $.
Any upper bound of the chromatic number of the smaller hypergraph $\hat H$ becomes an upper bound of the same for the bigger hypergraph $H$. We refer the reader to \cite{MR4452952,bretto2013hypergraph} and references therein for different upper bounds of chromatic numbers, which can be used as upper bounds for $\chi(\hat H)$. 
For any hypergraph $H$, a set $U\subseteq V(H)$ is said to be an \textit{independent set} if $U$ contains no hyperedge $e$.
If there exists an independent set $U$ in $V(H)$ then $H$ is $|V(H)|-|U|+1$-colourable.
 Thus, $ \chi(H)\le |V(\hat H)|-rank(\hat H)+2$, where $rank(H)=\max\limits_{e\in E(H)}|e|$. 
Since $\hat \pi:E(H)\to E(\hat H)$ is an onto map, we have the following Proposition.
\begin{prop}
	Let $H$ be a simple connected hypergraph $H$ with $|E(H)|>1$. For any independent set $U\subseteq V(\hat H)$ in $\hat H$, $\pi^{-1}(U)$ is an independent set in $H$. 
\end{prop}
 \begin{proof}  If possible let $\pi^{-1}(U)$ is not an independent set in $H$ then $e\subset \pi^{-1}(U)$ for some $e\in E(H)$. Thus, $\hat \pi(e)\subset U$ for $\hat\pi(e)\in E(H)$, a contradiction. Therefore, $\pi^{-1}(U)$ is independent set.    \end{proof}
In a simple connected hypergraph $H$ with $|E(H)|>1$, for each $a\in \mathfrak{C}(\mathfrak{U}(H))$, the set $\pi^{-1}(a)$ is an independent set. If $H$ is a simple hypergraph with the minimum partition number $m$ then for any $m-1$-tuple, $a_1,\ldots,a_{m-1}$, where $a_i\in \mathfrak{C}(\mathfrak{U}(H))$ for all $i=1,2,\ldots, m-1$, the set $\pi^{-1}(a_1)\cup\ldots\cup \pi^{-1}(a_{m-1})$ is an independent set.
\section{Discussion}\label{sec-discus}
Let $G$ and $H$ be two hypergraphs. A function $\mathfrak{f}:V(G)\to V(H)$ is called an hypergraph-isomorphism between $G$ and $H$ if $\mathfrak{f}$ is bijection and it induce another bijection $\hat{\mathfrak{f}}:E(H)\to E(H)$, defined as, $\hat{\mathfrak{f}}(e)=\{\mathfrak{f}(v):v\in e\}$ for all $e\in E(H)$. Thus, $E(H)=\{\hat{\mathfrak{f}}(e):e\in E(H)\}$.
A property of hypergraph $H$ is called a hypergraph-invariant if it is invariant under any hypergraph-isomorphism. For example, the order of the hypergraph $|V(H)|$, the number of hyperedges $|E(H)|$, the rank $r_H$ and corank $cr_H$.
Some hypergraph-invariant of a hypergraph $H$ involving our introduced building blocks are $|\mathfrak{U}(H)|$, $|V(\hat{H})|$, $|E(\hat H)|$, $|E(G_H)|$, $\max\limits_{W_{E_0}\in\mathfrak{U}(H)}|W_{E_0}|$, $\min\limits_{W_{E_0}\in\mathfrak{U}(H)}|W_{E_0}|$, maximum partition number, minimum partition number are hypergraph invariants. Another hypergraph invariant is the monotone non-decreasing sequence of unit-cardinalities of $H$, that is $\{|W_{E_1}|,\ldots,|W_{E_k}|\}$ where $\mathfrak{U}(H)=\{W_{E_1},\ldots ,W_{E_k}\}$ and $|W_{E_i}|\le |W_{E_{i+1}}|$ for all $i=1\ldots,k-1$.
For any hypergraph $H$, if $\mathfrak{C}(\mathfrak{U}(H))=\{a_1,a_2,\ldots,a_m\}$, and $a_i=\{W_{E^i_1},W_{E^i_2},\ldots,W_{E^i_{n_i}}\}$ then $|V(H)|=\sum\limits_{i=1}^m\sum\limits_{j=1}^{n_i}|W_{E^i_j}|$. If $|a_i|\le|a_{i+1}|$ then the monotone non-decreasing sequence $\{|a_1|,\ldots,|a_k|\}$ is an hypergraph invariant. The number of hyperedges $|E(H)|=\sum\limits_{\hat e\in E(\hat H)}|{\hat \pi}^{-1}(\hat e)|$. The number of connected components in $H$ is another hypergraph invariant, equal to the same $\hat H$. Here we have also introduced some distance-based invariants, namely, unit-radius, unit-diameter, unit-eccentricity, and unit-girth of a hypergraph.

A hypergraph symmetry(or hypergraph automorphism) of a hypergraph $H$ is a bijection $\phi$ on $V(H)$ that preserves the adjacency relation. That is, $e=\{v_1,\ldots,v_m\}\in E(H)$ if and only if $\hat\phi(e)=\{\phi(v_1),\phi(v_2),\ldots,\phi(v_m)\}\in E(H)$.

Total number of bijections on $W_{E_0}$ is $|W_{E_0}|!$. Each bijection $\phi:W_{E_0}\to W_{E_0}$, can be extended to an automorphism of $H$. 
Therefore, for any hypergraph $H$, each $W_{E_0}\in \mathfrak{U}(H)$ corresponds to $|W_{E_0}|!$ number of hypergraph automorphism.


If $W_{E_i}, W_{E_j}$ are twin units in $H$ with $|W_{E_i}|=|W_{E_j}|$ then each bijection $\phi:W_{E_i}\to W_{E_j}$ induces a hypergraph automorphism.

The above-mentioned hypergraph automorphisms 
alter the vertices only in a tiny part of $V(H)$ (namely, a unit or a pair of twin units). We refer to these automorphisms as local automorphisms.
The composition of a finite number of  local automorphisms is also a hypergraph automorphism. Evidently, for each $a\in\mathfrak{C}(\mathfrak{U}(H))$, the set $\pi^{-1}(a)$ corresponds to a local automorphism. 

The centrality of a vertex in a hypergraph $H$ is a function on $V(H)$ that determines the vertices' importance. 
The notion of centrality changes with the idea of importance. If the importance of a vertex depends on the star of the vertex, then the corresponding centrality function is constant on each unit. For example, the degree centrality of $v\in V(H)$ is the $|E_v(H)|$ constant on each unit.
In the case of disease propagation, each individual is considered as vertices, and interacting communities are the hyperedges. In this network, all the vertices in a unit are equally vulnerable to the disease. Thus, one can rank the units according to their vulnerabilities and take safety measures on the whole unit instead of acting on individuals. Scientists are vertices in the scientific collaboration network, and research projects are the hyperedges. The unit distance can be a good measure of collaborative distance.

\section{Conclusion}
For any equivalence relation $\mathfrak{R}$ on the vertex set, the $\mathfrak{R}$-equivalence class affects the spectra of any $\mathfrak{R}$-compatible operators.
Here, we have provided one example of equivalence relation $\mathfrak{R}_s$ such that the general signless Laplacian, general adjacency and general Laplacian are $\mathfrak{R}_s$-compatible. This equivalence relation $\mathfrak{R}_s$ leads us to the notion of units, and units lead us to twin units. 
Besides spectra of operators associated with a hypergraph, we have seen that they are also crucial for hypergraph automorphisms, hypergraph colouring, and random processes(like dynamical networks and random walks). Thus, finding other equivalence relations $\mathfrak{R}$ on the vertex set can help us identify more hypergraph building blocks.

In \Cref{1.unit}, we illustrate how the building blocks affect the structure of hypergraphs. The building blocks can be related to the hypergraph properties that are dependent on the hypergraph's structure. The spectra of hypergraphs, for example, are affected by the building blocks. We can also see that the building blocks are related to hypergraph colouring, hypergraph automorphisms, random walks on hypergraphs, and other applications. The study of building blocks has numerous applications. In dynamical networks on hypergraphs, for example, all dynamical systems in certain building blocks may exhibit similar behaviour. In the disease propagation model on hypergraphs, since all vertices in specific building blocks have the same expected hitting times, they form clusters of equally vulnerable people. We expect all nodes in the same unit to be informed simultaneously in the information spreading model. As  $\chi(H)\le \chi(\hat H)$, we can get sharper upper bounds of $\chi(H)$ using the existing upper bounds of  $\chi(\hat H)$. Since some building blocks are related to hypergraph automorphisms, we can use them to explore hypergraph symmetries. The concept of unit distance can be used to investigate collaborative distances in collaboration networks and some centralities in hypergraphs.

 \section*{Acknowledgement}
	The work of the author PARUI is supported by the University Grants Commission, India (Beneficiary Code/Flag: 	BININ00965055 A). 
 
\end{document}